\documentclass[final,onefignum,onetabnum]{siamart251216}

\usepackage{lipsum}
\usepackage{amsfonts}
\usepackage{graphicx}
\usepackage{epstopdf}
\usepackage{algorithmic}
\usepackage{subcaption}
\usepackage{siunitx}
\ifpdf
  \DeclareGraphicsExtensions{.eps,.pdf,.png,.jpg}
\else
  \DeclareGraphicsExtensions{.eps}
\fi

\newsiamremark{remark}{Remark}
\newsiamremark{hypothesis}{Hypothesis}
\crefname{hypothesis}{Hypothesis}{Hypotheses}
\newsiamthm{claim}{Claim}

\headers{Feynman's inverse problem}{A. Kirkeby}

\title{Feynman's inverse problem\thanks{Submitted to the editors 24.10.2023}}
\author{Adrian Kirkeby\thanks{Departmen of Numerical Analysis and Scientific Computing, Simula Research Laboratory, Oslo, 
  (\email{adrian@simula.no})}}

\usepackage{amsopn}

\makeatletter
\newcommand*{\addFileDependency}[1]{% argument=file name and extension
  \typeout{(#1)}% latexmk will find this if $recorder=0 (however, in that case, it will ignore #1 if it is a .aux or .pdf file etc and it exists! if it doesn't exist, it will appear in the list of dependents regardless)
  \@addtofilelist{#1}% if you want it to appear in \listfiles, not really necessary and latexmk doesn't use this
  \IfFileExists{#1}{}{\typeout{No file #1.}}% latexmk will find this message if #1 doesn't exist (yet)
}
\makeatother

\newcommand*{\myexternaldocument}[1]{%
    \externaldocument{#1}%
    \addFileDependency{#1.tex}%
    \addFileDependency{#1.aux}%
}
\ifpdf
\hypersetup{
  pdftitle={An Example Article},
  pdfauthor={Adrian Kirkeby}
}
\fi

\myexternaldocument{ex_supplement}

\begin{document}

\maketitle

% REQUIRED
\begin{abstract}
 We analyse an inverse problem for water waves posed by Richard Feynman in the BBC documentary \emph{Fun to Imagine}. We show that the problem can be modelled as an inverse Cauchy problem for gravity-capillary waves, conduct a detailed analysis of the Cauchy problem, and give a uniqueness proof for the inverse problem. Somewhat surprisingly, this results in a  positive answer to Feynman's question. In addition, we derive stability estimates for the inverse problem both for continuous and discrete measurements, propose a simple inversion method and conduct numerical experiments to verify our results. 
\end{abstract}

% REQUIRED

% REQUIRED
\section{Introduction}
Richard Feynman is one of the most iconic physicists of the 20th century and needs little introduction. In contrast to most earlier giants of science, Feynman was alive after the invention of the video camera, and this resulted in numerous entertaining and inspiring recordings of him talking about science and philosophy. In one such recording, the BBC documentary series \emph{Fun to Image} from 1983, there is a \href{https://www.youtube.com/watch?v=egB9p5ZbrEg}{segment}\footnote{It is recommended that you watch the clip before you go on reading the paper.} where Feynman talks about waves. We quote: 
\vspace{3mm}
\begin{center}
    \parbox{11cm}{\emph{If I'm sitting next to a swimming pool, and somebody dives in [...], I think of the waves and things that have formed in the water. And when there's lots of people who have dived in the pool there is a very great choppiness of all these waves all over the water. 
    And to think that it's possible, maybe, that in those waves is a clue to what's happening in the pool, that some sort of insect or something, with sufficient cleverness could sit in the corner of the pool and, just be disturbed by the waves, and by the nature of the irregularities and bumping of the waves, have figured out who jumped in where and when, and whats happening all over the pool.}}\\
\end{center}
\vspace{3mm}
With Feynman's characteristic enthusiasm, we are presented with an inverse problem for water waves, that is, a problem where one attempts to estimate some unknown information (``what's happening all over the pool") from an indirect measurement (``the irregularities and bumping of the waves" observed by the insect), causally related through the physics of water waves. 

Inverse problems are a central part of modern science, with applications in, for example, medical imaging, seismology, astronomy, and machine learning. Inverse problems are challenging, and they are often called ill-posed problems due to the possible lack of a unique solution and the instability of the inversion procedure. Indeed, it is often impossible to determine the exact cause of a measurement. However, inverse problems for some types of wave propagation, such as electromagnetic and acoustic waves, are among the more tractable examples, with a long list of successful applications \cite{scherzer2010handbook,mueller2012linear}. We are, however, dealing with a rather different type of waves, namely water waves. Again, Feynman in his ``Lecture on Physics, Vol. 1" gives the following characterization of water waves: 
\vspace{3mm}
\begin{center}
    \parbox{11cm}{\emph{Now, the next wave of interest, that are easily seen by everyone and which are usually used as an example of waves in elementary courses, are water waves. As we shall soon see, they are the worst possible example, because they are in no respects like sound and light: they have all the complications that waves can have.}}
\end{center}
\vspace{3mm}
In this paper we introduce Feynman's inverse problem for water waves. Using a linear two + three dimensional wave model (2D-3D)  wave model, we conduct a detailed analysis of the forward and inverse problem.
The question we want to address is ``How well does information propagate in water waves?". 
We show that the answer to this question is ``quite well" and explain some of the reasons. First, we give a positive answer to Feynman's question in the form of a uniqueness result (\Cref{sec:uniqueness}, \cref{uniqueness}). The answer is a bit surprising. It says that an observer measuring the wave amplitude and water velocity in any small area of the surface for an arbitrary short time can, in principle, determine the cause of the waves. The result relies on the non-locality of the system of partial differential equations (PDEs). 

In  \Cref{sec:stability} we use a spectral observability technique to obtain stability estimates for the inverse problem in the case where we have amplitude measurements of the waves along two adjacent sides of the pool, and in  \Cref{sec:stability-finite} we give a similar estimate for discrete measurements. Last, in  \Cref{section: experiments} we present a solution method for the inverse problem and test it on simulated data. \\
\newline 
Compared to the aforementioned acoustic or electromagnetic waves, inverse problems for water waves seem to have received less attention. However, there are some very interesting earlier works, and we now mention some of these. In \cite{smeltzer2019improved}, the authors use measurements of the surface wave to compute the near-surface currents, and in \cite{vasan2013inverse} and \cite{fontelos2017bottom} the authors aim to reconstruct the bottom topography from similar measurements. The paper \cite{sellier2016inverse} surveys several inverse problems related to free surface flows, and the topic of control and observability of both linear and non-linear water waves have been studied in, e.g., \cite{alazard2018control,reid1985boundary,reid1995control}. 

\section{Mathematical model}

\subsection{The forward problem}
We consider a rectangular pool of uniform depth. The geometry is then determined by the length $L$, the width $W$ and the depth $H$. We assume $W\leq L$. Let $X = (x_1,x_2)$. The domain occupied by the water is 
$$ \Omega_P = \left\{ (X,z) \in \mathbb{R}^3 :  0 < x_1 < L, 0 < x_2 < W, -H < z < 0   \right\}.$$
Next, we define
\begin{align*}
     \Gamma_P &= \left\{ (X,z) \in \overline{\Omega}_P\setminus \Omega_P :   -H \leq z < 0 \right\}, \\
    \Gamma_S &= \left\{X : 0 < x_1 < L, 0 < x_2 < W \right\},
\end{align*}
where $\overline{\Omega}_P$ denotes the closure of $\Omega_P$ and $\Gamma_P$ and $\Gamma_S $ represent the walls and bottom of the pool and the flat surface, respectively. \\

To define the PDE model, we let $\Delta_X = \partial^2_{x_1} + \partial^2_{x_2} $ and $\Delta_{X,z} = \partial^2_{x_1} + \partial^2_{x_2} + \partial^2_{z}$ be the 2D and 3D Laplacians, respectively, and let $\partial_\nu$ denote the outward pointing normal derivative on both $\Gamma_P, \Gamma_S$ and $\partial \Gamma_S$. 

We denote surface wave amplitude by $\eta(X,t)$, i.e.,  the height of the wave at the point $X \in \Gamma_S$ at time $t$ is ${z = \eta(X,t)}$.  
Next, we assume that the water is incompressible and irrotational, so the flow can be described in terms of a velocity potential $\phi(X,z,t)$. The velocity field is given by $V = \nabla_{X,z} \phi$, and the potential satisfies the Laplace equation $\Delta_{X,z} \phi = 0$. Since there is no flow of water through $\Gamma_P$, we impose a Neumann boundary condition on the potential,
\begin{equation*}
    V \cdot \mathbf{\nu} = \partial_\nu \phi   = 0 ,\quad (X,z) \in \Gamma_P, t > 0. 
\end{equation*}
\cref{fig:pool} illustrates the domain and quantities.

\begin{figure}[H]
    \centering
    \includegraphics[scale = 0.20]{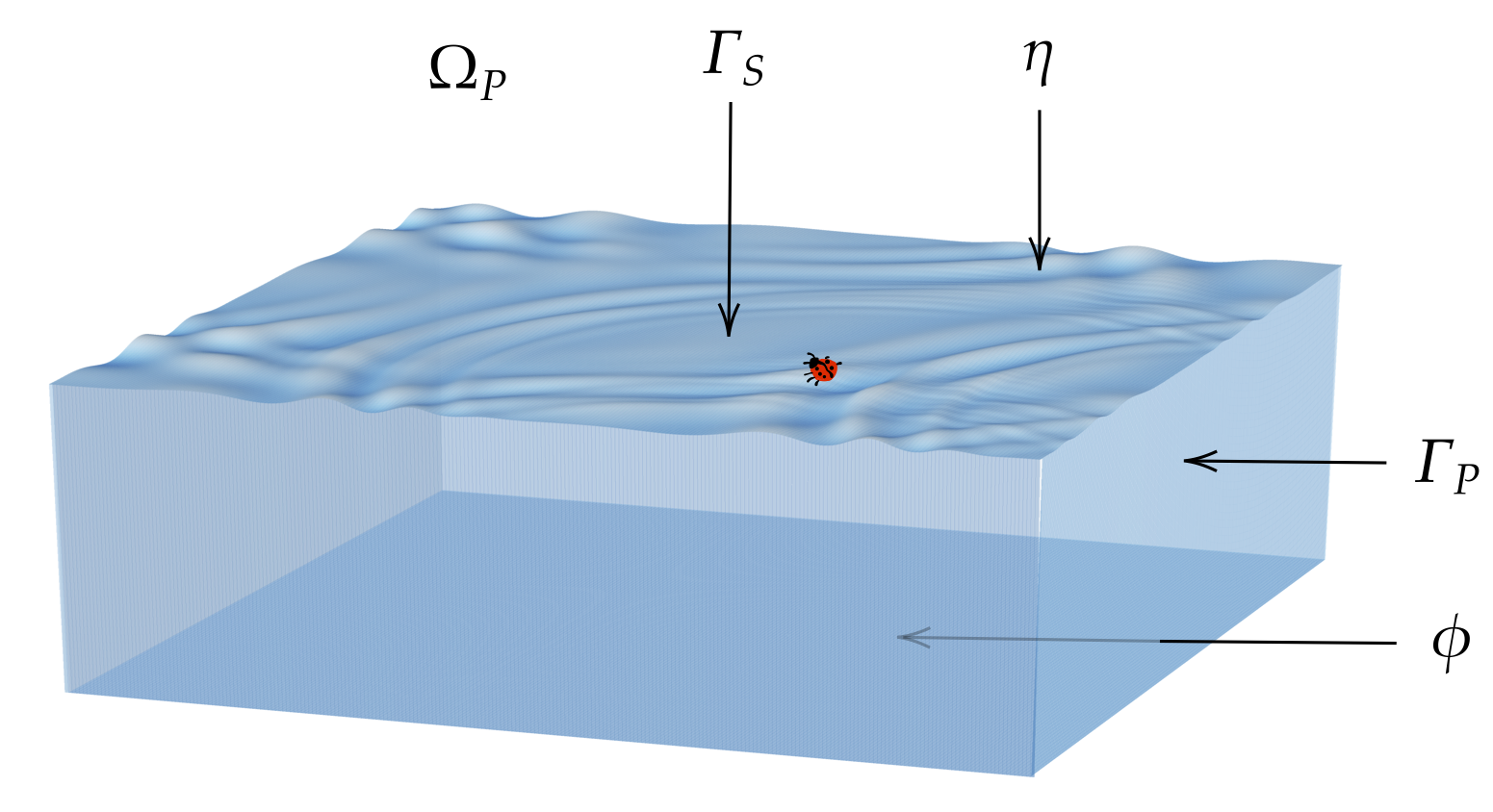}
    \caption{The domain $\Omega_P$. The surface wave amplitude $\eta$ is defined on $\Gamma_S$, while $\phi$ is defined in all of $\Omega_P$. The homogeneous Neumann condition is posed on the submerged pool boundary $\Gamma_P$. The question is if it is possible for the ladybug to figure out the initial state of this wave field from its local observation?}
    \label{fig:pool}
\end{figure}

We model the disturbance of the initially calm water, the ``lots of people who have dived in the pool", as the initial condition for the surface wave amplitude denoted by $\eta_0$, together with the initial condition for the velocity potential $\phi$ at the surface, denoted $\varphi_0$. To model the surface waves, we consider the Airy wave model with surface tension. This is a linearized wave model, and a derivation of it together with a discussion of its validity can be found in, e.g., Ch. 5 in \cite{ablowitz2011nonlinear}. Although one can show that surface tension has little influence on the behavior of longer waves (gravity waves), it is this effect that makes it possible for the insect to sit on the surface. Also, it becomes important when one considers shorter waves (capillary waves), and it turns out to be central to establish stability of the inverse problem. 

The small parameter in the linearization is the wave steepness $s = ka$, where $k$ is the magnitude of wave number and $a$ is the maximum amplitude, and the linearization is known to be a good approximation when  $s \ll 1$ and $a\ll H$ (cf. Ch. 1 in \cite{kuznet2002linear} and references therein). The system of PDEs is

\begin{align}
    \begin{cases}
      \qquad \qquad \partial_t \eta - \partial_z\phi\big|_{z=0} &=  \quad0, \quad  X \in \Gamma_S, t> 0, \\
  \partial_t \phi\big|_{z=0}  + g\eta  - S\Delta_X \eta &= \quad 0, \quad X \in \Gamma_S, t> 0,\\
     \qquad \qquad \qquad \quad   \Delta_{X,z} \phi &= \quad 0, \quad (X,z) \in \Omega_P, t > 0, \\
     \qquad \qquad \qquad \qquad \partial_\nu \phi  &= \quad  0,  \quad (X,z) \in \Gamma_P, t > 0, \\
     \qquad \qquad \qquad \qquad \partial_\nu \eta  &= \quad 0,  \quad X \in \partial \Gamma_S, t > 0, \\
     \qquad \left(\eta \big|_{t = 0}, \phi\big|_{z = 0, t = 0}\right) &= \quad \left(\eta_0,\varphi_0\right), \quad X \in \Gamma_S,
    \end{cases}
    \label{pool}
\end{align}
\newline 
where $g$ is the gravitational acceleration and $S = \frac{\sigma}{\rho}$, where $\sigma$ is the surface tension coefficient and $\rho$ is the mass density of water. 

The first equation in \eqref{pool} says that the velocity of the surface amplitude should equal the vertical velocity of the water at the surface, while the second equation results from linearizing the dynamic Bernoulli equation at the surface, with the pressure due to surface curvature, i.e., the surface tension,  given by $S\Delta_X \eta$. In the bulk of the water, $\Delta_{X,z}\Phi = 0$ ensures the incompressibility condition $\nabla_{X,z} \cdot V = 0$. The boundary condition on $\eta$ ensures that the waves are reflected at the edges of the pool, and the last equation is the initial condition.

\subsection{Measurements}
We consider several different measurements. Let $\Gamma_M \subset \Gamma_S$ be a subset of the surface, and for $0\leq t_0 < t_1$, let $I_T = [t_0,t_1]$ be the measurement time interval. In the analysis of the inverse problem for Feynman's insect, we consider the measurement $\mathcal{M}$ to be
\begin{equation*}
\mathcal{M}(X,t) = (\eta(X,t), \nabla_X \phi(X,0,t) , \quad X \in \Gamma_M, \quad    t \in I_T.
\end{equation*}
That is, we measure the wave amplitude $\eta$ and horizontal water velocity $(V_1,V_2) = \nabla_X \phi(X,0,t) $ on the set $\Gamma_M$, i.e., the ``the irregularities and bumping of the waves" observed by the insect on the surface.  
In \Cref{sec:stability} and \Cref{sec:stability-finite}, we assume that we only measure the amplitude $\eta$ and consider noisy measurements, measurements of the wave velocity, and discrete measurements.

\section{Analysis of the water wave model}
\label{section: forward problem analysis}
Before turning to the inverse problem, we introduce a convenient framework for the analysis and derive some useful properties of \eqref{pool}. For related work on linear (and non-linear) water waves, see, for example, \cite{johnson1997modern,bridges2016lectures,kuznet2002linear,craig2006surface,ablowitz2011nonlinear}.
\subsection{Preliminaries}
\label{section: preliminaries}
First, note that for conservation of mass to hold in the pool, we must have  $\text{vol}(\Omega_P) = \int_{\Gamma_S}(H+\eta(X,t)) \mathrm{d}X $ and hence $\int_{\Gamma_S}\eta(X,t)\mathrm{d}X  = 0, \forall t \geq 0.$ This motivates the use of the function space 
$$\dot{L}^2= \left\{ u \in L^2(\Gamma_S) : \int_{\Gamma_S} u \mathrm{d}X = 0 \right\},$$
equipped with the $L^2$-inner product. Note that 
${L^2(\Gamma_S)  = \{\text{constant functions} \} \cup \dot{L}^2(\Gamma_S)}$, and so $\dot{L}^2$ is a Hilbert space.
Next, we use the functions $$\psi_{m,n}(X) = C_{m,n}\cos(\pi m x_1/L)\cos(\pi n x_2/W), \hspace{0.5em}C_{m,n} = \begin{cases}
    \frac{\sqrt{2}}{\sqrt{LW}}, \quad \text{if } m \text{ or } n = 0, \\
    \frac{2}{\sqrt{LW}}, \quad \text{otherwise, }
\end{cases}\hspace{0.5em} $$ with $m,n\in \mathbb{N}_0.$ One can verify that $\{\psi_{m,n}\}_{m+n > 0}$ is an orthonormal basis for $\dot{L^2}$. We write $\hat{u}_{m,n} = (u,\psi_{m,n})_{L^2}$, so $u \in \dot{L}^2$ has the expansion $u = \sum_{m+n > 0} \psi_{m,n}\hat{u}_{m,n}.$ We also introduce the Sobolev spaces $\dot{H}^s$, 
\begin{equation*}
    \dot{H}^s = \left\{u\in \dot{L}^2 : \sum_{m + n > 0}  (m^2 + n^2)^s|\hat{u}_{m,n}|^2 < \infty \right\}, \quad s \in \mathbb{R}. 
\end{equation*}
with the inner product $(u,v)_{\dot{H}^s} = \sum_{m+n > 0}(m^2 + n^2)^s \hat{u}_{m,n}\hat{v}_{m,n}.$ \\

We can simplify equation \eqref{pool} by considering the trace of the velocity potential at the surface.  Let $\varphi(X,t) = \phi(X,z,t)|_{z = 0}$. To reduce \eqref{pool} to an equation for $(\eta,\varphi)$, we introduce the Dirichlet-to-Neumann (DN) operator $\mathcal{G}$, 

\begin{equation}
    \mathcal{G}\varphi = \partial_z \phi|_{z=0}, \quad \text{where } \phi \text{ solves} \quad  \begin{cases}\Delta_{X,z} \phi = 0, \quad \text{in} \quad \Omega_P, \\
     \partial_\nu \phi  = 0, \quad \text{on} \quad   \Gamma_P, \\
      \phi\big|_{z = 0} = \varphi, \quad \text{on} \quad \Gamma_S.
    \end{cases}
    \label{d2n}
\end{equation}
Furthermore, we write  $\mathcal{L} = (S\Delta_X - g) $. We can now rewrite \eqref{pool} as 
\begin{equation}
    \frac{\mathrm{d}}{\mathrm{d}t}\begin{pmatrix} \eta \\\varphi \end{pmatrix} = \begin{bmatrix}
        0& \mathcal{G} \\ \mathcal{L} &  0\end{bmatrix} \begin{pmatrix} \eta \\\varphi \end{pmatrix}, \quad \begin{pmatrix} \eta \\\varphi \end{pmatrix}_{t=0} = \begin{pmatrix} \eta_0 \\\varphi_0 \end{pmatrix}.
        \label{craig-sulem}
\end{equation}
This is known as the Zakharov--Craig/Sulem formulation (cf. \cite{waterwavesprob}). It reduces the 2D-3D-system \eqref{pool} to a 2D-2D-system for $(\eta,\varphi)$, at the expense of introducing the non-local DN operator. In  \Cref{section: proofs}, we prove some useful properties of $\mathcal{G}$ and $\mathcal{L}$ using their spectral representation in the basis $\{\psi_{m,n}\}$. 

\subsection{Well-posedness of waves in the pool}
\label{section: wellposed}
We introduce the energy space $H_E$ for the system \eqref{craig-sulem}. 
We write $$ A= \begin{bmatrix}
        0& \mathcal{G} \\ \mathcal{L} &  0\end{bmatrix} \quad \text{and} \quad U = \begin{bmatrix}
            \eta \\
            \varphi
        \end{bmatrix},$$
We set $H_E = \dot{H}^1\times \dot{H}^{1/2}$, equipped with the inner product 
\begin{align*} 
    \langle U,V \rangle_E & = S(\nabla u_1,\nabla v_1)_{L^2} + g(u_1,v_1)_{L^2} +(\mathcal{G}u_2,v_2)_{L^2}, \quad U = \begin{pmatrix} u_1 \\ u_2 \end{pmatrix}, V = \begin{pmatrix} v_1 \\ v_2 \end{pmatrix}  \in H_E. 
\end{align*}
It follows from the considerations in \Cref{section: proofs} that this is a proper inner product, and that $D(A) = \dot{H}^3 \times \dot{H}^{3/2} $ is a suitable domain for A, i.e., $A: D(A) \to H_E$ is continuous. The $H_E$ norm is a physical energy norm since the term $S(\nabla \eta,\nabla \eta)_{L^2} $ represents the elastic energy due to the surface tension and the terms $g(\eta,\eta)_{L^2}$ and $ (\mathcal{G}\varphi,\varphi)_{L^2} = \int_{\Omega_P} V^2 \mathrm{d}X$ are proportional to the potential and kinetic energy of the water, respectively. Next, note that 
\begin{align*}
    \langle U,V \rangle_E &= S(\nabla u_1,\nabla v_1)_{L^2} + g(u_1,v_1)_{L^2} +(\mathcal{G}u_2,v_2)_{L^2} \\
     &= -(u_1,\mathcal{L}v_1)_{L^2} + (u_2,\mathcal{G}v_2)_{L^2}, \quad U = \begin{pmatrix} u_1 \\ u_2 \end{pmatrix}, V = \begin{pmatrix} v_1 \\ v_2 \end{pmatrix}  \in D(A). 
\end{align*}
Since $\dot{U} = AU$, we see that $\|U\|_E^2 = \langle U,U \rangle_E^2 $ is conserved. Assuming $U$ solves \eqref{craig-sulem}, we have
\begin{align*}
    \frac{\mathrm{d}}{\mathrm{d}t}\left(\frac{1}{2}\|U\|_E^2\right) = \langle \dot{U},U \rangle_E  = -(\mathcal{G}u_2,\mathcal{L}u_1)_{L^2} + (\mathcal{L}u_1,\mathcal{G}u_2)_{L^2} = 0. 
\end{align*}

We are now in a position to show well-posedness of \eqref{craig-sulem}. The proof, which follows the semigroup approach used in \cite{reid1995control}, is found in \Cref{section: proofs}. 
\begin{theorem}
    For each $U_0 = (\eta_0,\varphi_0) \in H_E$, the water waves system \eqref{craig-sulem} has a unique solution $U = (\eta,\varphi)^\top \in C([0,\infty);H_E)$ that satisfies $\|U(t)\|_E = \|U_0\|_E $ for all $t \geq 0$. In addition,
    \begin{equation}
         U(t) = \sum_{m,n \geq 0 } \cos(\omega_{m,n}t)\psi_{m,n}\begin{bmatrix}(\eta_0,\psi_{m,n})_{L^2} \vspace{3mm}\\ (\varphi_0,\psi_{m,n})_{L^2} \end{bmatrix} 
         + \sin(\omega_{m,n}t)\psi_{m,n}\begin{bmatrix} \frac{\omega_{m,n}}{g + Sk^2_{m,n}}(\varphi_0,\psi_{m,n})_{L^2} \vspace{3mm}\\ -\frac{g + Sk_{m,n}^2}{\omega_{m,n}}(\eta_0,\psi_{m,n})_{L^2}\end{bmatrix},
        \label{solution}
    \end{equation}
where 
\begin{equation}
    \omega_{m,n} = \sqrt{(gk_{m,n} + Sk_{m,n}^3)\tanh(k_{m,n}H)})  \quad \text{with} \quad  k_{m,n} = \pi\sqrt{(m/L)^2 + (n/W)^2}
    \label{dispersion}
\end{equation} 
is the dispersion relation for gravity-capillary waves. Moreover, for $(\eta_0,\varphi_0) \in \dot{H}^{s +\frac{1}{2}} \times \dot{H}^{s}$ with $s \geq 1/2,$ we have that 
$$ \eta \in \dot{H}^{s +\frac{1}{2}} \quad \text{and} \quad \varphi \in \dot{H}^{s} \quad \text{for all } t \geq 0. $$

\label{existence}
    
\end{theorem}
\emph{Note:} To ease notation later, the sum above includes $(m,n) = (0,0)$, although the term is always zero since $(\psi_{0,0},f)_{L^2} = 0$ for $f \in \dot{L}^2$.

With the well-posedness of the forward problems established, we are ready to tackle the inverse problem.  

\begin{figure}[H]
     \centering
     \begin{subfigure}[b]{0.45\textwidth}
         \centering
         \includegraphics[scale = 0.23]{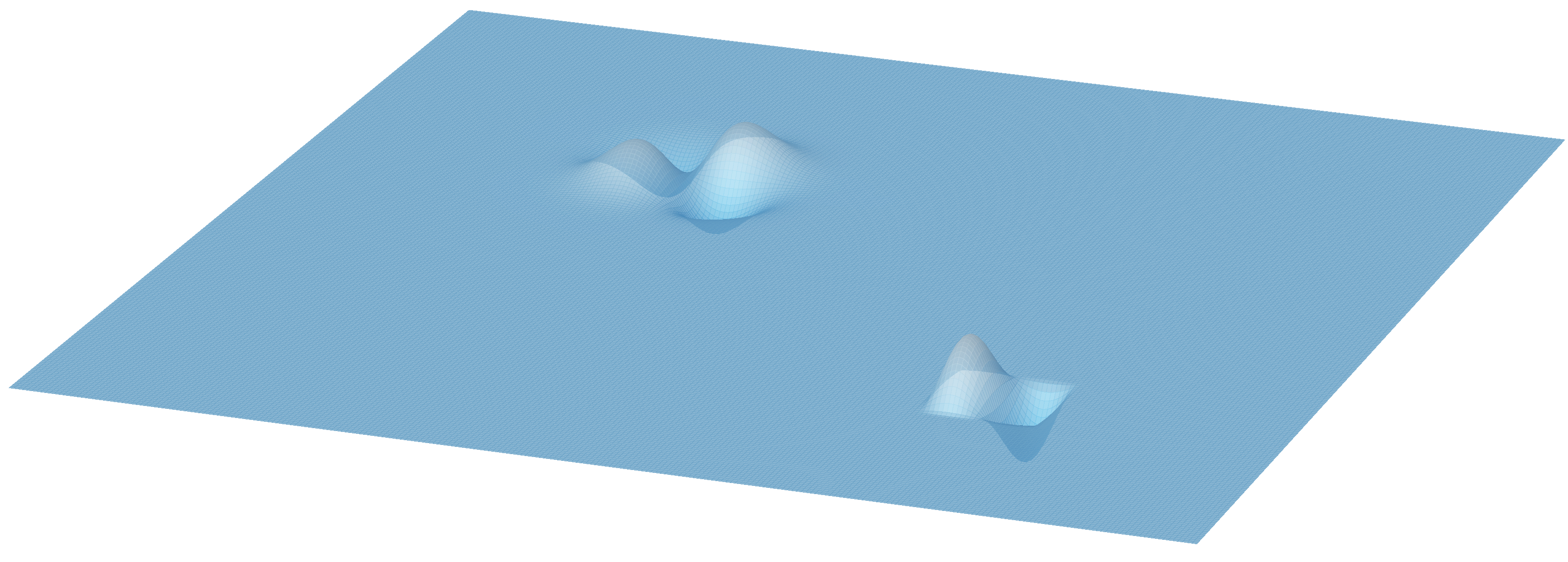}
         \caption{$t=0$}
         \label{fig:t=0}
     \end{subfigure}
     \hspace{5mm}
\begin{subfigure}[b]{0.45\textwidth}
         \centering
         \includegraphics[scale = 0.23]{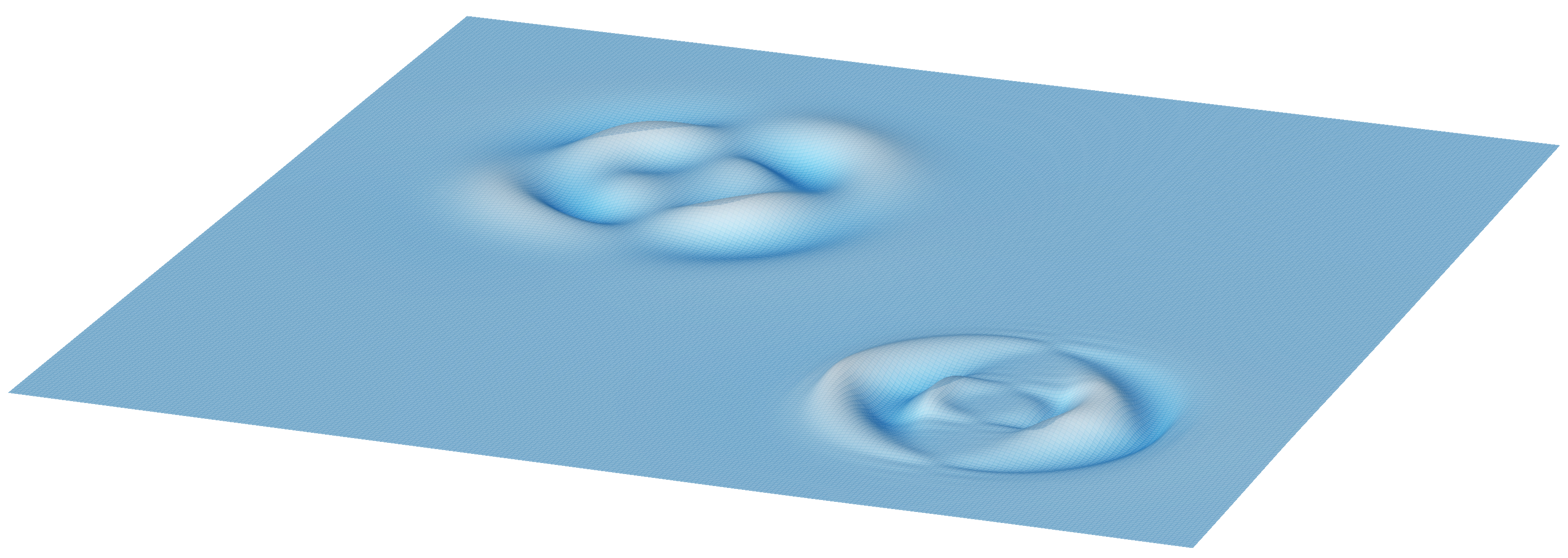}
         \caption{$t = 0.5$}
         %\label{fig:three sin x}
     \end{subfigure}
     \hspace{5mm}
 \begin{subfigure}[b]{0.45\textwidth}
         \centering
         \includegraphics[scale = 0.23]{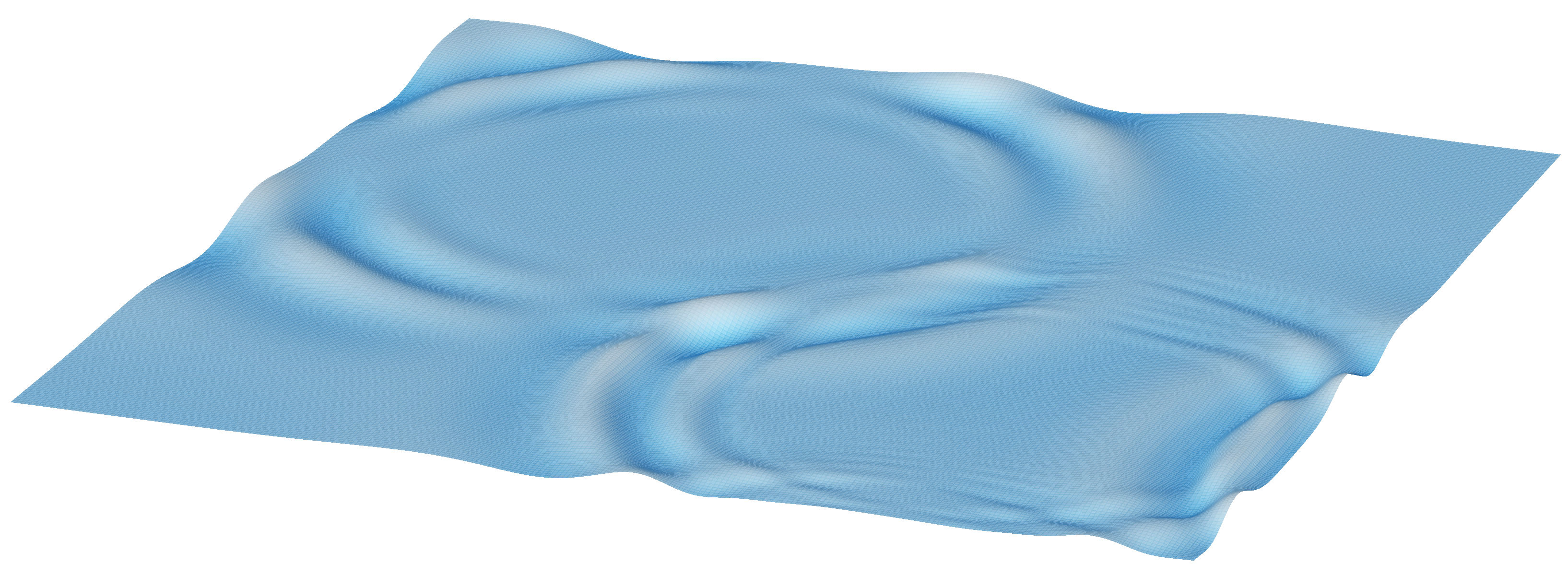}
         \caption{$t = 1.5$}
         %\label{fig:five over x}
     \end{subfigure}
     \hspace{5mm}
     \begin{subfigure}[b]{0.45\textwidth}
         \centering
         \includegraphics[scale = 0.23]{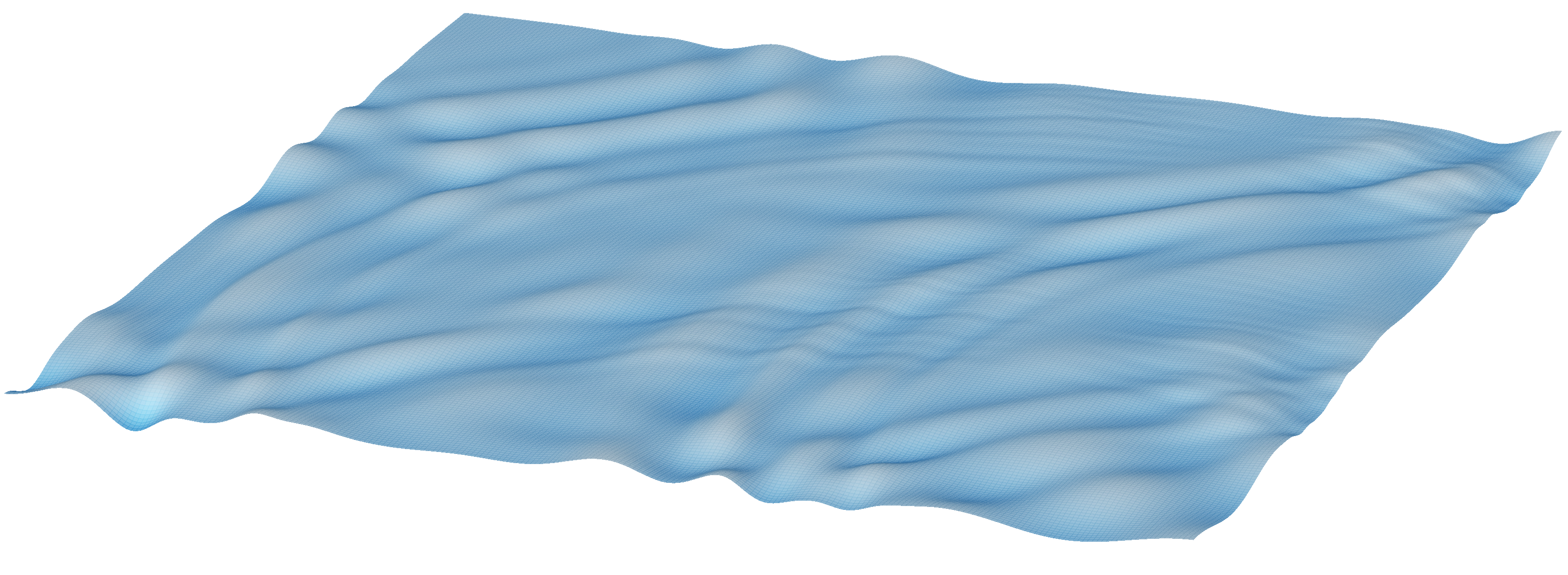}
         \caption{$t=3.5$}
         %\label{fig:five over x}
     \end{subfigure}
        \caption{The figure shows four snapshots of the wave amplitude $\eta$, with initial conditions and parameters as given in \Cref{sec:numerical experiments}.} 
        \label{fig:wave sim}
\end{figure}

\section{The inverse problem}
We now analyze the inverse problem for several different measurements. We consider the following problems:
 \begin{itemize}
     \item[1:] Determination of the initial disturbance from measurement on a small domain. 
     \item[2:] Stability of reconstructing the initial disturbance from measurements along the sides of the pool.  
     \item[3:] Stability of reconstruction of the initial disturbance from discrete measurements. 
 \end{itemize}

\subsection{Problem 1: Determination of the initial disturbance for Feynman's clever insect}
\label{sec:uniqueness}
We assume the measurement domain $\Gamma_M $ is a non-empty, open subset of the surface $\Gamma_S$, and let $I_T =[t_0,t_1]$ be the measurement time interval. Let the measurement be $\mathcal{M}(X,t) = (\eta(X,t),\nabla_X\varphi(X,t))$ with $ (X,t) \in \Gamma_M\times I_T$.    
\begin{theorem}
    Assume $(\eta_0,\varphi_0) \in \dot{H}^{\frac{3}{2}} \times \dot{H}^1$. Then the measurement $\mathcal{M}(X,t)$ uniquely determines the initial disturbance $(\eta_0,\varphi_0)$. 
    \label{uniqueness}
\end{theorem}
For waves in the linear regime, the answer to Feynman's question is therefore positive.
\begin{corollary}
A sufficiently clever insect can determine the initial disturbance in the pool from the measurement $\mathcal{M}$. 
\end{corollary} 
\begin{remark}
    The above result might seem strange when considering the limitations to inverse problems for other types of waves, e.g., acoustic and electromagnetic waves. For the two latter types of waves, information propagates at a finite speed. Hence, for any finite measurement time, an initially localized disturbance will not even be detectable by an observer positioned sufficiently far away.
    
    Mathematically, this is seen by considering the Cauchy problem for the wave equation in $\mathbb{R}^n$,  
    $$(\partial_t^2 - c^2\Delta) u = 0, \quad (u(0),\partial_t u(0)) = (u_0,u_1).$$ 
    This is equation describes the propagation of acoustic and electromagnetic waves in a homogeneous medium with wave speed $c$. Huygens' principle says that if $u_0(X) = u_1(X) =  0$ outside some ball $B_R(0)$ with radius $R$, then $u(t,X)) = 0 $ outside of the ball $B_{R + ct}$ for all $t \geq 0$ (cf. Ch. 6 in \cite{craig2018course}).

    The result also has the strange consequence that if the water is still on some open set on the surface for a tiny duration of time, then the whole body of water must have been still everywhere and for all time. 
    
    As becomes clear in the proof, it is the non-local behavior of the DN operator that is the mechanism enabling this local uniqueness property; due to the modeling of the fluid as incompressible and irrotational, the velocity potential is harmonic. This allows for the application of the unique continuation property, but it is still surprising that this property propagates enough information to ensure the unique determination of the initial disturbance. For a physical intuition of this effect, one can think of stepping into a bathtub; as one descends into the water, the water level appears to increase by the same amount everywhere and at once. Apparently, the local disturbance manifests itself globally and instantly.  
    
    We note that the unique continuation property has been established for various water wave models before (cf. \cite{kenig2020unique,kenig2020uniqueness,zhang1992unique}), and for the Schrödinger equation, which is also a dispersive PDE \cite{tenenbaum2009fast}. Moreover, there has recently been much interest in inverse problems related to fractional PDEs \cite{ruland2020fractional,kaltenbacher2023inverse}. Fractional differential operators are non-local operators and for some of the inverse problems considered for such operators, unique continuation properties and local uniqueness results can be established (cf. \cite{covi2023uniqueness,ghosh2020calderon,ghosh2020uniqueness}).
\end{remark}

\begin{proof}
The proof relies on the energy estimate in  \cref{existence} and the unique continuation property from Cauchy data for elliptic PDEs, and we paraphrase here Corollary 2.22 in \cite{choulli2016applications}: 
\emph{Let $\Omega \subset \mathbb{R}^n$ be a Lipschitz domain and assume $\Delta u = 0$ in $\Omega$. Let $\Gamma$ be a non-empty open subset of $\partial \Omega$. Then $u|_\Gamma = \partial_\nu u|_\Gamma = 0$ implies that $ u = 0 \text{ in } \Omega$.} \\

     By \cref{existence}, the measurement is well-defined in the sense that $\int_{\Gamma_S} \mathcal{M} h \mathrm{d}X$ is bounded for all $h \in C^\infty(\Gamma_S)$. Now, assume $\mathcal{M}(X,t) = 0$ for $(X,t) \in \Gamma_M\times I_T$. Then $\dot{\eta} = 0$ in some space-time cylinder $Q_\varepsilon = B_\epsilon \times I_\varepsilon  \subset \Gamma_M \times I_T$, and by \eqref{craig-sulem} this implies that $\mathcal{G}\varphi = 0 $ and $\dot{\varphi}= 0$ in $Q_\varepsilon$.  Moreover, since $\nabla_X \varphi = 0$ in $Q_\varepsilon$, we must have $\varphi|_{Q_\varepsilon} = C_\varepsilon$ for some constant $C_\varepsilon$. But this implies $\phi = C_\varepsilon$. Assume there is some $\tilde{\phi} \neq C_\varepsilon$, and such that $\tilde{\phi}|_{Q_\varepsilon} = C_\varepsilon $ and $\partial_z \tilde{\phi}|_{Q_\varepsilon} = 0$. Then $\tilde{\phi} - C_\varepsilon$ is harmonic in $\Omega_P$ and satisfies ${(\tilde{\phi} - C_\varepsilon)|_{Q_\varepsilon} = \partial_z(\tilde{\phi} -C_\varepsilon)|_{Q_\varepsilon} = 0}$. We can now apply the unique continuation result and conclude that  $\tilde{\phi} = C_\varepsilon$ in $\Omega_P$. Therefore  $\varphi = C_\varepsilon$, and since $\varphi \in \dot{L}^2 $, $\varphi = 0$. Hence, with $\dot{\varphi} = 0 $, we have that $(S\Delta_X - g)\eta = 0$ on $\Gamma_S$. But since $g/S > 0$, the only solution to $-\Delta_X \eta = -\frac{g}{S}\eta$ on $\Gamma_S$ with homogeneous Neumann conditions is $\eta = 0$. By \cref{existence} we have $  \|(\eta(X,t),\varphi(X,t))\|_E = \|(\eta_0,\varphi_0)\|_E $ for all $t \geq 0$ and hence $(\eta_0,\varphi_0) = 0$, and since the map $(\eta_0,\varphi_0) \mapsto \mathcal{M}$ is linear, the conclusion now follows.  
\end{proof}

\Cref{uniqueness} result says nothing about the stability of the procedure of inversion. In fact, unique continuation is known to be a severely ill-posed problem \cite{elcrat2012stability}. Hence, an observer who makes imperfect measurements would most likely not succeed with the task of computing the initial conditions. In the next section, we investigate the stability properties of the inverse problem. 

\subsection{Problem 2: Uniqueness and stability from measurements along two adjacent sides of the pool}
\label{sec:stability}

From the study of the observability of wave equations, it is well known that the part of the boundary where one measures the wave should satisfy certain geometric conditions to yield stability and uniqueness \cite{bardos1992sharp,observation}. Here, we utilize the explicit solution in \cref{existence} and a so-called spectral observability method used in  \cite{tenenbaum2009fast,komornik2014cross,observation}. 
Assuming that we measure the wave amplitude along two adjacent sides of the pool, i.e., $$ \mathcal{M}(X,t) = \eta(X,t)|_{X \in \Gamma_M }, \quad \text{with }   \Gamma_M = \left\{(x_1,0): x_1 \in (0,L) \right\} \cup \left\{(0, x_2): x_2 \in (0,W) \right\},   $$
we obtain the following result: 
\begin{theorem}
Assume the measurement time satisfies 
\begin{equation}
    T \geq  \frac{6L\sqrt{L^{2} + W^{2}} }{ W}\left(\frac{Hg}{2\sqrt{1+\frac{H^2g}{3S}}} - \frac{3S}{2H}\right)^{-1/2}.
    \label{Tbound}
\end{equation}
Then there is a constant $C$, independent of $(\eta_0,\varphi_0)$, such that following stability estimate holds 
\begin{equation}
    \int_0^T \int_{\Gamma_M} |\mathcal{M}(X,t)|^2 \mathrm{d}X\mathrm{d}t \geq C T \left(\|\eta_0\|_{L^2}^2 + \|\varphi_0\|_{\dot{H}^{-\frac{1}{2}}}^2\right).
\end{equation}
   \label{stability} 
\end{theorem}

\begin{remark} 
The term 
$$\tiny{\left(\frac{Hg}{2\sqrt{1+\frac{H^2g}{3S}}} - \frac{3S}{2H}\right)^{-1/2}}$$
is a lower bound on the group velocity $c_g = \partial_k \omega(k)$ of the water waves. The group velocity is the velocity with which a localized wave group/wavelet propagates, and it is also the velocity of energy transport (cf. \cite{ablowitz2011nonlinear}). The lower bound on the observation time $T$ is given by a term proportional to the maximal length of the pool times the minimal group velocity. This is comparable to similar bounds for acoustic or electromagnetic waves, where the bound is proportional to the minimal wave velocity times the maximal distance (cf. \cite{alberti2018lectures}).  
Moreover, as we see in the proof of \cref{stability}, the presence of surface tension and the resulting capillary waves is necessary to have stability. In fact, for pure gravity waves ($S=0$ in \eqref{dispersion}), the group velocity $c_g(k) = \partial_k \omega(k) \sim k^{-1/2}$ approaches zero as the wavenumber $k$ increases. Consequently, one can have a compactly supported disturbance comprised of modes with frequencies so high that they do not reach the boundary during the measurement time $T \geq 0$. In contrast, this cannot happen for gravity-capillary waves, as $c_g(k) \sim k^{1/2}$. \\
\end{remark}
Equation \eqref{stability} implies uniqueness and stability. If we have a noisy measurement $\Tilde{\mathcal{M}} = \mathcal{M} +\varepsilon$ that we attribute to some initial data $(\tilde{\eta}_0,\tilde{\varphi}_0)$, it shows that the difference between $(\tilde{\eta}_0,\tilde{\varphi}_0)$ and the true initial data $(\eta_0,\varphi_0)$ is bounded by $\varepsilon$, i.e.,  
$$\int_0^T \int_{\Gamma_M}|\varepsilon|^2 \mathrm{d}X\mathrm{d}t \geq C\frac{T}{L\pi} \left(\|\eta_0 - \tilde{\eta}_0\|_{L^2}^2 + \|\varphi_0-\tilde{\varphi}_0\|_{\dot{H}^{-\frac{1}{2}}}^2\right).$$

\subsection*{Additional stability}
 The more precise measurements one is able to make, the more accurate the reconstruction should be. For example, if we can measure the amplitude with such precision that we can compute the (vertical) wave velocity without too much error, or if we are able to measure the velocity or acceleration directly, this information should result in a better reconstruction. We can quantify this by expressing the stability estimate with norms that include derivatives. As a simple consequence of \cref{stability} we get the following corollary.

\begin{corollary}
     Assume $k \in \mathbb{N}_0$, $(\eta_0,\varphi_0) \in \dot{H}^{\frac{3}{2}k}\times \dot{H}^{\frac{3}{2}k -\frac{1}{2}}$, and  that $T$ satisfies the same bound as in \cref{stability}. Then there is a constant $C_k$, independent of $(\eta_0,\varphi_0)$, such that following stability estimate holds
    \begin{equation}
    \int_0^T \int_{\Gamma_M} |\partial_t^k\eta(X,t)|^2 \mathrm{d}X\mathrm{d}t \geq C_k T (\|\eta_0\|_{\dot{H}^{\frac{3}{2}k}}^2 + \|\varphi_0\|_{\dot{H}^{\frac{3}{2}k-\frac{1}{2}}}^2).
\end{equation}
\label{increased stability}
\end{corollary}

Proofs of the above results are found in \Cref{section: proofs} and are based on an analysis of the dispersion relation and application of non-harmonic Fourier analysis.

\subsection{Problem 3: Discrete measurements and bandlimited reconstructions}
\label{sec:stability-finite}
We now investigate the more practical question of reconstruction from discrete, noisy boundary measurements. By discrete measurements, we mean that the wave amplitude is measured at a finite number of points in time and space, in contrast to the continuum case considered in \cref{uniqueness} and \cref{stability}. This is always the case in any real-world experiment or application, and was recently studied in, e.g., \cite{alberti2022infinite,harrach2021uniqueness}. The experimental setup we consider could be a water tank with transparent walls and a camera recording the waves as they hit the wall. The resulting measurement has the form
$$\{\eta(x_i,0,t_j) : 0\leq i \leq N, 0 \leq j \leq M\}, $$
for some finite $N,M \in \mathbb{N}$. Even in the absence of noise from such measurements, one can in general not obtain reconstructions of an arbitrary initial disturbance. We instead aim to reconstruct finite-dimensional disturbances, and introduce the space 
\begin{equation*}
    \dot{H}_\beta = \{ u \in \dot{L}^2 : u \in \text{span}\{\psi_{m,n} : 0 \leq m,n \leq \beta\} \}.
\end{equation*}
The space $\dot{H}_\beta$ is a space of bandlimited functions, i.e., functions that are finite linear combinations of the basis functions $\psi_{m,n}$, and where the maximal spatial frequency is given by the bandwidth $\beta$. For $f \in \dot{L}^2$, we denote by $P_\beta f$ the projection of $f$ on to $\dot{H}_\beta$, i.e., 
$$P_\beta f = \sum_{m,n \leq \beta} (f,\psi_{m,n})_{L^2} \psi_{m,n}(X).$$

Consider now a measurement along one side of the pool, $\Gamma_M = \{(x_1,0) : 0 \leq x_1 \leq L\}$, and assume $\eta_0 \in \dot{H}^s, s \geq 1$ and $ \varphi_0 = 0$. 
Given $\beta > 0$, we sample $\eta$ as follows: Let $N_x = \beta +1$ and set $x_k = (k+1/2)\frac{L}{N_x}$. For some $T$ and $N_t$ to be determined, set $ \Delta t = T/N_t$ and $t_j = j\Delta t $. We introduce two different measurements: 
\begin{equation}
\begin{split}
     \mathcal{M} = \left[\eta(x_k,0,t_j)\right] + \varepsilon_{k,j}, \\
    \mathcal{M}_\beta = \left[P_\beta \eta(x_k,0,t_j)\right],
\end{split} 
    \qquad  0 \leq k \leq N_x-1, \quad 0\leq j \leq N_t. 
    \label{M-sample}
\end{equation}
Here $\mathcal{M}$ is the noisy measurement of the true wave, while $\mathcal{M}_\beta$ is the exact measurement of $\eta$ with initial condition $P_\beta \eta_0 \in \dot{H}_\beta$, i.e., with a finite-dimensional initial condition. The noise is modelled by the terms $\varepsilon_{j,k}$ which are specified in the numerical experiment section. We use the Frobenius norm $\|\mathcal{M}\|_F^2 =\sum_{k,j}|m_{k,j}|^2$ as a norm on the measurement, and we denote by $R_\beta \eta_0$ the reconstruction of $\eta_0$ from the measurement $\mathcal{M}$.

Given a bandwidth $\beta$, the following theorem shows how to choose $T$ and $N_t$ such that we can stably approximate $P_\beta \eta_0$ from the measurements $\mathcal{M}$. The proof, found in \Cref{section: proofs}, is constructive and leads to a simple reconstruction method. 

\begin{theorem}
Choose $\beta \in \mathbb{N}$, and let 
 
 $$\delta = \frac{4\pi}{5 \omega_{\beta,\beta}}\min\left(1/5, \frac{\pi}{L}\left(\sqrt{\beta^2 + 1} -\beta\right)\sqrt{\frac{Hg}{2\sqrt{1+\frac{H^2g}{3S}}} - \frac{3S}{2H}}\right).$$
 Take $N_t \in \mathbb{N}$ such
 $$ N_t > \max\left(2\beta +2,\frac{1}{\delta}\right)  \quad \text{and} \quad T = \frac{4\pi N_t}{5\omega_{\beta,\beta}}.$$

Let $\mathcal{M}$ and $ \mathcal{M}_\beta$ be as described in \eqref{M-sample}, and assume $\eta_0 \in \dot{H}^s, s \geq 1$. Then the following estimates holds: 

\begin{equation*}
    \| P_\beta \eta_0 - R_\beta \eta_0\|_{L^2} \leq  \sqrt{\frac{LW}{2N_x\left(N_t-\frac{1}{\delta}\right)}}\|\mathcal{M}_{\beta} - \mathcal{M}\|_F, 
\end{equation*}
\begin{equation*}
    \| \eta_0 - R_\beta \eta_0\|_{L^2} \leq  \frac{2}{\beta^s}\|\eta_0\|_{\dot{H}^s} + \sqrt{\frac{LW}{2N_x\left(N_t-\frac{1}{\delta}\right)}}\|\mathcal{M}_{\beta} - \mathcal{M}\|_F. 
\end{equation*}

 \label{stability-finite}
\end{theorem}
\begin{remark} For a given bandwidth $\beta$, the above result quantifies the reconstruction error in terms of how much the true measurement $\mathcal{M}$ deviates from the the ideal measurement $\mathcal{M}_\beta$ and in terms of how well $\eta_0$ is approximated by $P_\beta \eta_0$. Moreover, one can infer that as the bandwidth $\beta$ increases, longer observation time and higher sampling frequency is needed for the stability estimate to hold. This makes sense, since as $\beta$ increases, we need to ``filter out" frequencies that are closer together. However, as we mention in \Cref{section: experiments}, the estimates one gets from the above theorem on $N_t$ and $T$ are pessimistic, and it turns out that one also can get good reconstructions from using fewer samples. 
\end{remark}

\section{Inversion - method and numerical experiments}

We now propose a reconstruction method for the measurement setup in \cref{stability-finite} and test it on synthetic data.

\subsection{Numerical solution of the forward problem}
\label{num_sol}
To generate synthetic data, we compute a high precision approximation to $\mathcal{M}$ from the solution formula \eqref{solution} and add noise to the measurement. 
\begin{itemize}
    \item[1:] Choose the number $N_a$ of eigenfunctions included in the solution \eqref{solution}, so that $N_a \gg \beta$, 
    where $\beta$ is the bandwidth of the reconstruction. 
    \item[2:] For initial data $\eta_0$, compute the coefficents $q_{m,n}  = (\eta_0,\psi_{m,n})_{L^2}$ using a suitable numerical quadrature method. 
    \item[3:] Given the sampling points $(x_k,t_j)$, let
    $$ \mathcal{M}(k,j) = \sum_{m,n \leq N_a}\cos(\omega_{m,n}t_j)\psi_{m,n}(x_k,0)q_{m,n}, \quad 0 \leq k \leq N_x-1, \quad  0 \leq j \leq N_t.$$
    \item[4:] Add noise. Let $N_\varepsilon$ be the relative noise level, i.e., $$ \frac{\|\varepsilon\|_F}{\|\mathcal{M}\|_F }\leq N_\varepsilon.$$ We assume that the noise is $\mathbf{\varepsilon} \sim \mathcal{N}(0,I_{N_xN_t})$, i.e., independent and identically distributed Gaussian random variables. Given a realization of $\varepsilon'$ of the noise, we make sure the relative error is achieved by setting $\varepsilon = \varepsilon' N_\varepsilon\|\mathcal{M}\|_F/\|\varepsilon'\|_F$ and take the noisy measurement to be $\mathcal{M}^\varepsilon = \mathcal{M} + \varepsilon$.  
\end{itemize}

\subsection{Inversion method and numerical experiments}
\label{section: experiments}
Based on the the proof of \cref{stability-finite}, we propose the following inversion method. It uses the discrete cosine matrix 
\begin{equation*}
    D_3 = [ \cos(\pi(k+1/2)m/N_x) ], \quad 0 \leq k \leq N_x-1, \quad  0\leq m \leq N_x-1.
\end{equation*}
and the non-harmonic Vandermonde matrix 
\begin{equation*}
    V_m(j,:) = 
    \begin{bmatrix}
        (\overline{v_\beta^m})^j \quad   \cdots \quad  (\overline{v_0^m})^j \quad (v_0^m)^j  \quad \cdots \quad (v_\beta^m)^j
    \end{bmatrix},
     \quad \text{for } 0 \leq j \leq N_t, 
\end{equation*}
where $V_m(j,:)$ denotes the $j'$th row of the matrix $V_m$, $v_n^m =e^{i\omega_{m,n}\Delta t}$, $\Delta T = T/N_t$ and $\beta$ is the bandwidth.  We also denote by $\mathcal{C}$ the matrix of the Fourier coefficients of $P_\beta \eta_0$ (should you want to implement the method yourself, we suggest you skim through the proof of \cref{stability-finite} for more details). 
\subsubsection*{Reconstruction method (Non-harmonic Fourier inversion)}
\begin{itemize}
    \item[1:] Choose a bandwidth $\beta $. Given the problem parameters $L,W,H,g$ and $S$, compute the minimal frequency gap $\delta$ and choose the sampling parameters $N_x,N_t$ and $T$ according to \cref{stability-finite}. 
    \item[2:] Obtain the measurement $\mathcal{M}^\varepsilon$, i.e., the matrix of wave amplitude samples.
    \item[3:] Compute $U = D_3^{-1}\mathcal{M}$. For $m = 0,1,2,...,\beta$, solve the linear system 
    $$ \mathcal{C}^\top(m,:) = (V_m^\dagger V_m)^{-1}V_M^\dagger U(m,:)^\top,$$
    where $V_m^\dagger$ denotes the Hermitian transpose of $V_m$. 
    \item[4:] Extract coefficients $\{p_{m,n}\}$ from $\mathcal{C}$ and set $R_\beta \eta_0 = \sum_{m,n\leq \beta}p_{m,n}\psi_{m,n}(X).$
\end{itemize}

\subsubsection*{Numerical experiments}
\label{sec:numerical experiments}
\Cref{tab:numerical parameters} summarises the physical parameters used in the experiments. For the initial condition we choose
\begin{align*}
    \eta_0(X) &= -\frac{0.1}{\sigma^2}(x_1 - x_1^g)(x_2 - x_2^g)\exp\left(-\frac{(x_1 -x_1^g)^2 +(x_2 - x_1^g)^2}{\sigma^2}\right) \\
        & + 0.05\chi_{x_1^s,x_2^s,I}\sin( \pi(x_1 - (x_1^s -I/2))/I)\sin( 2\pi(x_2 - (x_1^s/2-I/2)/I)) - C_0,
\end{align*}
where $\sigma = 0.05, (x_1^g,x_2^g) = (L/3,W/3), (x_1^s,x_2^s) = (3L/4,3W/4), I = L/10$, $\chi_{x_1^s,x_2^s,I}$ is the indicator function for the square with side length $I$ centered at $(x_1^s,x_2^s)$, and $C_0$ is a constant such that $\int_{\Omega} \eta_0 \mathrm{d}X = 0$.

\begin{table}[H]
    \centering
    \begin{tabular}{ccccc}
        g & S & L & W & H\\ \hline
        $9.81 \si{m/s^2}$  & $7.08 \times 10^{-5} \si{m^3/s^2}$ & $1 \si{m}$ & $1 \si{m}$ & $0.5 \si{m} $\\
    \end{tabular}
    \caption{}
    \label{tab:numerical parameters}
\end{table}
We take the relative noise level to be $ N_\varepsilon = 0.1$ (i.e., $10 \%)$. Furthermore, we take three different values of the bandwidth parameter $\beta$ and conduct the inversion for each one. The sampling parameters for the different values of $\beta$ are shown in \Cref{tab:sampling parameters}. 
\begin{table}[H]
    \centering
    \begin{tabular}{c|cccc}
        $\beta$ & $N_t$ & $T$ & $\Delta t$ & $N_x$ \\ \hline
        8  &   770  &  103 \si{s} & 0.134 \si{s} & 9 \\
        16 \  & 2979 & 278 \si{s}  & 0.093 \si{s} & 17 \\
        32  & 11090  & 697 \si{s} & 0.062 \si{s} & 33  
        \end{tabular}
     \caption{}
    \label{tab:sampling parameters}
\end{table}
Here we see that the number of samples in time $N_t$ and the sampling time $T$ increases substantially as a function of $\beta$. 
As the conditions in \cref{stability-finite} are only sufficient, we in addition carry out the inversion for each given $\beta$ with $N_t/2$ samples (but $\Delta t$ as in \Cref{tab:sampling parameters}).   

\Cref{fig:wave sim} shows the initial condition and some snapshots the resulting wave produced by the method outlined in  \Cref{num_sol}, while  \cref{fig:measurement} shows a section of a measurement $\mathcal{M}$.
\begin{figure}
    \centering
    \includegraphics[width = 1\textwidth]{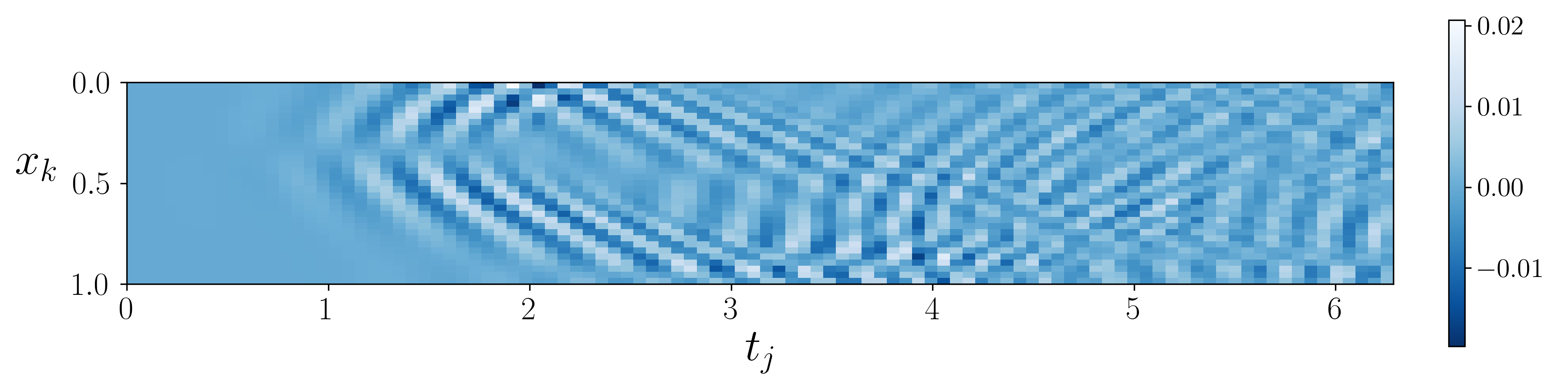}
    \caption{The figure is a plot of the sampled data in $\mathcal{M}$. The pixel $\mathcal{M}(k,j)$ is the measurement value at position $x_k$ and time $t_j$. }
    \label{fig:measurement}
\end{figure}

\subsection{Results}

 \Cref{tab:relative errors,tab:errors thm 5 b} give the relevant measures of error in the reconstructions $R_\beta \eta_0$, and in \Cref{fig:reconstruction errors} we show the reconstructions $R_\beta \eta_0$ and the pointwise error. 
We can see that the quality of the reconstruction, both pointwise and in the $L^2$-norm, goes from quite poor to quite good as we increase $\beta$. Moreover, we see that the inequalities in \cref{stability-finite} holds true, but that they are not very sharp. In a real experiment, the lack of any dissipation mechanism in \eqref{pool} makes the model unrealistic, as there would probably be very few waves bouncing around in a small wave tank after the 10 minute measurement time required for \cref{stability-finite} to hold. However, the relative errors on inversion with $T/2$ (as compared to the values in \cref{tab:sampling parameters}) are almost identical to those with $T$, indicating the possibility for improving the bounds on the sampling parameteres.  
\begin{table}[H]
    \centering
    \begin{tabular}{c|ccc}
        $\beta$ & $\frac{\|\eta_0 - R_\beta \eta_0\|_{L^2}}{\|\eta_0\|_{L^2}} $  & $\frac{\|P_\beta\eta_0 - R_\beta \eta_0\|_{L^2}}{\|P_\beta \eta_0\|_{L^2}} $ & $\frac{\|\eta_0 - P_\beta \eta_0\|_{L^2}}{\|\eta_0\|_{L^2}}$ \\  \hline 
        8  &    0.9329  &  0.0989  & 0.9322   \\
        16 \  & 0.5769 & 0.0177  & 0.5767  \\
        32  & 0.111  & 0.008  & 0.110  
        \end{tabular}
    \caption{}
    \label{tab:relative errors}
\end{table}

\begin{table}[H]
    \centering
    \begin{tabular}{c|cc}
        $\beta$ & $\|P_\beta\eta_0 - R_\beta \eta_0\|_{L^2} $ & $ \sqrt{\frac{LW}{2(N_t - 1/\delta)N_x}}\|\mathcal{M}_{\beta} - \mathcal{M}\|_F$  \\  \hline 
        8  &   0.00014  &    0.0227    \\
        16 \  & 0.00005 &   0.0270  \\
        32  & 0.000033  &  0.0107
        \end{tabular}
    \caption{}
    \label{tab:errors thm 5 a}
\end{table}
\begin{table}[H]
    \centering
    \begin{tabular}{c|cc}
        $\beta$ & $\|\eta_0 - R_\beta \eta_0\|_{L^2} $  & $\frac{\sqrt{2}}{\beta}\|\eta_0\|_{\dot{H}^1} + \sqrt{\frac{LW}{2(N_t - 1/\delta)N_x}}\|\mathcal{M}_{\beta} - \mathcal{M}\|_F $ \\  \hline 
        8  &   0.0037  &  0.0709    \\
        16 \  & 0.0023 & 0.0433   \\
        32  & 0.0005  &  0.0218
        \end{tabular}
    \caption{}
    \label{tab:errors thm 5 b}
\end{table}

\begin{figure}[H]
    \centering
    \includegraphics[width = 0.95\textwidth]{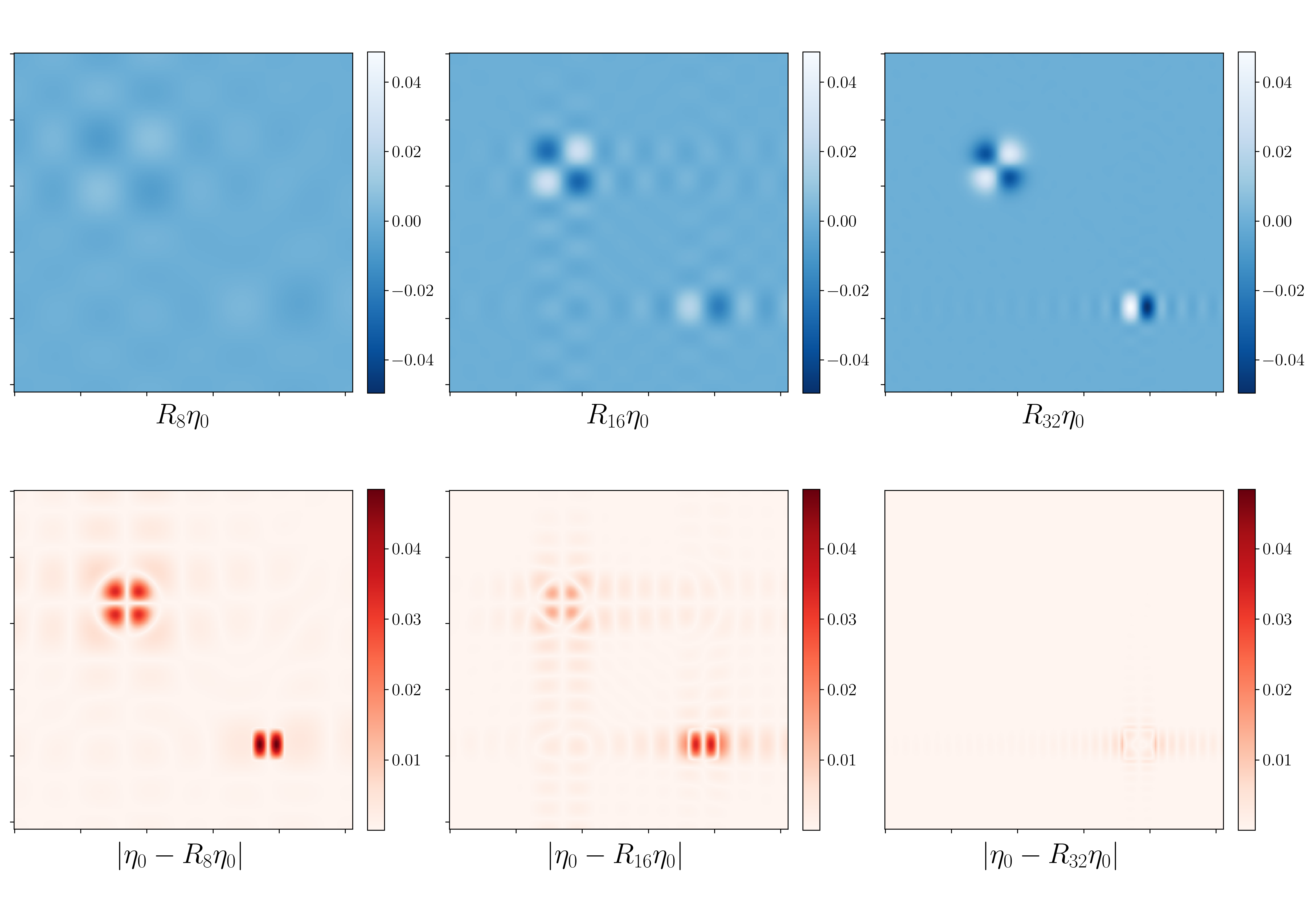}
    \caption{Plot of the reconstruction $R_\beta \eta_0$ and pointwise error $|R_\beta \eta_0 - \eta_0|$ for various $\beta$. The results are in line with the estimates in \cref{stability-finite}, and the quality of the reconstruction increases substantially with each doubling of $\beta$.}
    \label{fig:reconstruction errors}
\end{figure}

\section{Conclusion}
\label{section = conclusion}
In this paper we have analyzed our proposed model of Feynman's inverse problem. The main result is the unique determination of the initial disturbance from measurement on an open set, giving a positive answer to Feynman's question. We have also shown that due to the presence of surface tension, one can obtain stability results similar to those that hold for acoustic and electromagnetic waves, from boundary measurements of water waves. Moreover, we have proposed and numerically tested a reconstruction method applicable to discrete, noisy boundary measurements.  

As water waves models are mathematically very rich (cf. \cite{waterwavesprob,ablowitz2011nonlinear}), there are many more complicated, non-linear models that can be used to learn more about the propagation of information by water waves (or surface waves, in general). Also, one could consider compressible fluid models or a situation where the disturbance is time-dependent (e.g., moving around). Moreover, the analysis shows that water waves contain a lot of information about their previous state, and so it might be possible to use a similar approach to study conditions leading to rouge waves or other rare surface wave phenomena. \\

We conclude by going back to the BBC recording. Feynman continues by saying: 
\vspace{3mm}
\begin{center}
    \parbox{11cm}{\emph{And that's what we're doing when we're looking at something. The light that comes out is waves, just like in the swimming pool, except in three dimensions instead of the two dimensions of the pool [...] And we have an eight of an inch black hole, into which these things go [...] It's quite wonderful that we figure out so easy... that's really because the light waves are easier than the waves in the water, [they are] a little bit more complicated. It would have been harder for the bug than for us, but it's the same idea: to figure out what we are looking at at a distance.}}
\end{center}
\vspace{3mm}
We can reasonably say, at least for waves in a truly incompressible fluid in the linear regime, that Feynman was not entirely correct in his argument:  
It might be harder for the insect to figure out what is happening in the pool, but it is, in theory, possible. And had the insect, in contrast, been observing electromagnetic waves like the eye does (or had the water waves behaved like the electromagnetic waves), this would not be true. 

\section{Proofs}
\label{section: proofs}
This section contains the proofs that were not in the main sections and short introductions to the techniques used. 
\subsection{Proof of \cref{existence}}

\subsection*{The operators $\mathcal{G}$ and $\mathcal{L}$}
We first consider the DN operator as defined in \eqref{d2n}. It differs slightly from the standard definition, as it is usually defined on the whole boundary (cf. \cite{mueller2012linear}, Ch. 12). 
Due to the simple geometry of the pool, we apply separation of variables to find that the unique solution $\phi$ in \eqref{d2n} is 
\begin{equation}
    \phi(X,z) = \sum_{m+n > 0} \psi_{m,n}(X)\frac{\cosh(k_{m,n}(z+H))}{\sinh(k_{m,n}H)} \tanh(k_{m,n}H)(\varphi,\psi_{m,n})_{L^2(\Gamma_S)},
\end{equation}
where $k_{m,n} = \pi\sqrt{(m/L)^2 + (n/W)^2}$ and $\psi_{m,n}$ are the basis functions defined in  \Cref{section: preliminaries}. Hence 
\begin{equation} \mathcal{G}\varphi = \frac{\partial}{\partial z} \phi|_{z= 0} = \sum_{m+n > 0} k_{m,n}\tanh(k_{m,n}H)(\varphi,\psi_{m,n})_{L^2(\Gamma_S)}\psi_{m,n}(X),
\label{G}
\end{equation}
and we have expressed $\mathcal{G}$ in terms of the Fourier multiplier $|D|\tanh(|D|H)$. Since we can find constants $a_1,a_2 > 0$ such that $$a_1|m^2 +n^2|^{1/2} \leq k_{m,n}\tanh(k_{m,n}H) \leq a_2 |m^2 + n^2|^{1/2},$$ we infer that 
$\mathcal{G} : \dot{H}^s \to \dot{H}^{s-1}$. Moreover, $\mathcal{G}$ is symmetric on $\dot{H}^s$, and 
$(\mathcal{G}u,v)_{L^2}$ is an inner product on $\dot{H}^{1/2}$. Also, $\mathcal{G}$ is one-to-one, and since for any $v \in \dot{H}^{s-1}$, we may set 
$$u = \sum_{m+n > 0}  \frac{\hat{v}_{m,n}}{k_{m,n}\tanh(k_{m,n}H)}\psi_{m,n} $$
and compute 
$$\|u\|_{\dot{H}^s}^2 = \sum_{m+n > 0}  \frac{ (m^2 + n^2)^s|\hat{v}_{m,n}|^2}{(k_{m,n}\tanh(k_{m,n}H))^2} \leq C\sum_{m+n > 0} (m^2 + n^2)^{s-1}|\hat{v}_{m,n}|^2 = C \|v\|_{\dot{H}^{s-1}}^2,$$
it follows that $\mathcal{G} : \dot{H}^s \to \dot{H}^{s-1}$ is onto, and so it is bijective. 

Furthermore, ${\mathcal{G} : D(\mathcal{G}) \to \dot{H}^{1/2}, D(\mathcal{G}) =\dot{H}^{3/2} }$ is self-adjoint. This follows from the fact that for  
any linear operator $T$ on $\dot{H}^{1/2}$ that is one-to-one and has a dense range, it holds that 
$(T^*)^{-1} = (T^{-1})^*$ (cf. Ch. 10, \cite{kreyszig1991introductory}), where $T^*$ denotes the adjoint, i.e., $(Tu,v)_{\dot{H}^{1/2}} = (u,T^*v)_{\dot{H}^{1/2}}$. Using this, it is straight forward to verify that with $T = \mathcal{G}^{-1}$ we get that $(\mathcal{G}^{*})^{-1}= \mathcal{G}^{-1} $
and so $\mathcal{G}^{*} = \mathcal{G}$.
%%{$\mathcal{G} : \dot{H}^s \to \dot{H}^{s-1}$, $\mathcal{G}$ is invertible and $(\mathcal{G}u,v)_{L^2}$ is an inner product on $\dot{H}^{1/2}$. That $\mathcal{G}:\dot{H}^{1/2} \to \dot{H}^{-1/2} $ is self-adjoint now follows, since by \eqref{G} we have that $(\mathcal{G}u,v)_{L^2} = (u,\mathcal{G}v)_{L^2}$ for all $u,v \in \dot{H}^{1/2}$.}
%%Moreover, $\mathcal{G} : D(\mathcal{G}) \to \dot{H}^{1/2}, D(\mathcal{G}) =\dot{H}^{3/2} $ is self-adjoint. This follows from the fact that for a 
%%any linear operator $T$ on $\dot{H}^{1/2}$ that is injective and has a dense range, it holds that 
%$(T^*)^{-1} = (T^{-1})^*$ (cf Ch. 10, \cite{kreyszig1991introductory}). Using this, it is straight forward to verify that with $T = \mathcal{G}^{-1}$ we get that $(\mathcal{G}^{*})^{-1}= \mathcal{G}^{-1} $
%and so $\mathcal{G}^{*} = \mathcal{G}$.

The boundary value problem associated with $\mathcal{L}$ (and $\eta$) is
\begin{equation*}
    \begin{cases}
        \mathcal{L}u = f,  \quad \text{in} \quad \Gamma_S, \\
        \partial_\nu u = 0  \quad \text{on} \quad \partial \Gamma_S.
    \end{cases}
    \label{neumann}
\end{equation*}
It has a unique solution $u \in \dot{H}^1$ for every $f \in \dot{L}^2$ (cf. Ch. 8 in \cite{salsa2016partial}), and $u$ satisfies 
\begin{equation*}
     S(\nabla u,\nabla v)_{L^2} + g(u,v)_{L^2} = -(f,v)_{L^2}, \forall v\in \dot{H}^1.
    \label{neumann2}
\end{equation*} 
We also have that $$\mathcal{L}\psi_{m,n} = -(g + Sk_{m,n}^2)\psi_{m,n}, \quad m,n \in \mathbb{N}_0,$$
and so
\begin{equation*}
     u = \mathcal{L}^{-1}f = \sum_{m+n>0} \frac{\hat{f}_{m,n}}{-(g + Sk_{m,n}^2)}\psi_{m,n}.
     \label{Linv}
\end{equation*} 
By the same arguments used for $\mathcal{G}$, we can therefore conclude that $\mathcal{L}: \dot{H}^3 \to \dot{H}^1 $ is a bijection and that $\mathcal{L}$ is self-adjoint on $\dot{H}^1$.  

We now have what we need to prove well-posedness for the forward problem. 

\begin{proof}
(Proof of \cref{existence})
In \Cref{section: forward problem analysis}, we wrote the water waves system on the form
$$ \dot{U} = AU, \quad U(0) = U_0, \quad \text{with} \quad  A= \begin{bmatrix} 
        0& \mathcal{G} \\ \mathcal{L} &  0\end{bmatrix} \quad \text{and} \quad U = \begin{bmatrix}
            \eta \\
            \varphi
        \end{bmatrix}.$$
We first show that $A : D(A) \to H_E$, where $D(A) = \{\dot{H}^3\} \times \dot{H}^{3/2}$ is the domain of $A$, is skew-adjoint. $A$ is skew-symmetric, since 
$$ \langle A U,V \rangle_E = - \langle U,AV \rangle_E , \quad \forall U,V \in D(A). $$
Since both $\mathcal{L}$ and $\mathcal{G}$ are bijections, we have $R(A) = H_E$, where $R(A)$ denotes the range of $A$. We can therefore conclude that $A$ is skew-adjoint (cf. Ch. 3.7 in \cite{observation}).
By Stone's theorem (Theorem 3.8.6, \cite{observation}), $A$ is the generator of a strongly continuous, unitary semigroup $\mathbb{T}_t$. Hence
the unique solution to the initial value problem $\dot{U} = AU, U(0) = U_0$ in \eqref{craig-sulem} is given by $U(t) = \mathbb{T}_tU_0$, and $U \in C([0,\infty);H_E)$ satisfies $\|U(t)\|_E = \|U_0\|_E$.

Moreover, $A$ is diagonalizable. Indeed, one can check that 
$$ A V_{m,n}^\pm = \pm i \omega_{m,n} V_{m,n}^\pm, $$
where 
$$ V_{m,n}^\pm = \begin{bmatrix} (2(g + Sk_{m,n}^2))^{-1/2} \vspace{3mm}\\ \pm i(2k_{m,n} \tanh(k_{m,n}H))^{-1/2}\end{bmatrix}\psi_{m,n} \quad \text{and} \quad  \omega_{m,n} = \sqrt{(gk_{m,n} + Sk_{m,n}^3)\tanh(k_{m,n}H)}.$$
The eigenvectors here are chosen such that $\{V^+_{m,n}\}$ and $\{V^-_{m,n} \}$ are biorthogonal sequences, i.e., $\langle V_{m,n}^+,V_{i,j}^-\rangle_E = \delta_{m,n,i,j}.$ Consequently, 

$$ U(t) = \sum_{m,n \geq 0} e^{i\omega_{m,n}t}\langle U_0,V^-_{m,n} \rangle_E V_{m,n}^+ + e^{-i\omega_{m,n}t}\langle U_0,V^+_{m,n} \rangle_E V^-_{m,n}.$$
Since 
$$e^{i\omega_{m,n}t}\langle U_0,V^-_{m,n} \rangle_E V_{m,n}^+  + e^{-i\omega_{m,n}t}\langle U_0,V^-_{m,n} \rangle_E V^+_{m,n} = 2\text{Re}(e^{i\omega_{m,n}t}\langle U_0,V^-_{m,n} \rangle_E V^+_{m,n}), $$ 
and 
$$ \langle U_0,V^-_{m,n} \rangle_E V_{m,n}^+ =\frac{1}{2}\begin{bmatrix}  (\eta_0,\psi_{m,n})_{L^2} - i \frac{(k_{m,n} \tanh(k_{m,n}H))^{1/2}}{(g + Sk_{m,n}^2)^{1/2} }  (\varphi_0,\psi_{m,n})_{L^2} \vspace{3mm}\\
 (\varphi_0,\psi_{m,n})_{L^2} + i \frac{(g + Sk_{m,n}^2)^{1/2} }{(k_{m,n} \tanh(k_{m,n}H))^{1/2}}(\eta_0,\psi_{m,n})_{L^2}
\end{bmatrix}\psi_{m,n},
$$
it follows that 

\begin{equation*}
    U(t) = \sum_{m,n \geq 0} \cos(\omega_{m,n}t)\psi_{m,n}\begin{bmatrix}(\eta_0,\psi_{m,n})_{L^2} \vspace{3mm}\\ (\varphi_0,\psi_{m,n})_{L^2} \end{bmatrix} + \sin(\omega_{m,n}t)\psi_{m,n}\begin{bmatrix} \frac{\omega_{m,n}}{g + Sk^2_{m,n}}(\varphi_0,\psi_{m,n})_{L^2} \vspace{3mm}\\ -\frac{g + Sk_{m,n}^2}{\omega_{m,n}}(\eta_0,\psi_{m,n})_{L^2}\end{bmatrix}.
\end{equation*}
Last, note that for $m+n > 0$, we can find positive constants such that 
\begin{equation*}
\begin{split} c_1|m^2 + n^2|^{-1/4} &\leq  \frac{\omega_{m,n}}{g + Sk^2_{m,n}} \leq c_2|m^2 + n^2|^{-1/4}, \\
\tilde{c}_1|m^2 + n^2|^{1/4}  &\leq \frac{g + Sk_{m,n}^2}{\omega_{m,n}} \leq \tilde{c}_2|m^2 + n^2|^{1/4}. 
\end{split}
\end{equation*}
Assuming $(\eta_0,\varphi_0) \in \dot{H}^{s +\frac{1}{2}} \times \dot{H}^{s} $, we see that
\begin{align*}
    \|\eta(X,t) \|_{\dot{H}^{s+\frac{1}{2}}} &\leq \|\eta_0 \|_{\dot{H}^{s+\frac{1}{2}}} + \left\|\sum_{m,n} \frac{\omega_{m,n}}{g + Sk^2_{m,n}}(\varphi_0,\psi_{m,n})_{L^2} \psi_{m,n} \right\|_{\dot{H}^{s+\frac{1}{2}}} \\
    &\leq C\left(\|\eta_0 \|_{\dot{H}^{s+\frac{1}{2}}} + \|\varphi_0 \|_{\dot{H}^{s}} \right), \\
    \|\varphi(X,t) \|_{\dot{H}^{s}} &\leq \|\varphi_0 \|_{\dot{H}^{s}} + \left\|\sum_{m,n} \frac{g + Sk^2_{m,n}}{\omega_{m,n}}(\eta_0,\psi_{m,n})_{L^2}\psi_{m,n} \right\|_{\dot{H}^{s}} \\
    &\leq \tilde{C}\left(\|\varphi_0 \|_{\dot{H}^{s}} + \|\eta_0 \|_{\dot{H}^{s+\frac{1}{2}}} \right). \\
\end{align*}

\end{proof}

\subsection{Proofs of \Cref{stability} and \cref{increased stability}}

\subsubsection*{Measurements and non-harmonic Fourier series}
The measurement operator takes the form 
\begin{equation}
    \mathcal{M}(X,t) = \sum_{m,n \geq 0}^\infty  \psi_{m,n}(X)\left(\cos(\omega_{m,n}t) q_{m,n} + \sin(\omega_{m,n}t)\frac{\omega_{m,n} }{g + Sk_{m,n}^2}p_{m,n}\right),
    \label{measurement1}
\end{equation}
with  $  q_{m,n} = (\eta_0,\psi_{m,n})_{L^2(\Gamma_S)}$ and $ p_{m,n} = (\varphi_0,\psi_{m,n})_{L^2(\Gamma_S)}$.
Setting \begin{equation*}
    c_{m,n} =  \frac{1}{2}\left(q_{m,n} - i \frac{\omega_{m,n}}{g + Sk_{m,n}^2}p_{m,n}\right), \quad \text{for} \quad  m,n \geq 0, 
\end{equation*}
 we may write  
\begin{equation*}
    \mathcal{M}(X,t) = \sum_{m \geq 0 }\left(\sum_{n\geq 0}^\infty  \psi_{m,n}(X) c_{m,n} e^{i\omega_{m,n}t} + c.c.\right), 
\end{equation*}
where $c.c.$ denotes the complex conjugate of $\sum_{n\geq 0}^\infty \psi_{m,n}(X) c_{m,n} e^{i\omega_{m,n}t}$. 
Fixing $m $ and temporarily ignoring the spatial dependence in the above expression, we have a function of the form  
\begin{equation}
    f_m(t) = \sum_{n \geq 0} c_{m,n} \mathrm{e}^{i\omega_{m,n} t} + c.c.. 
      \label{Mexp}
\end{equation}
This is an example of a non-harmonic Fourier series (or almost periodic functions) \cite{young2001introduction,observation}.  They are functions on the form 
$$
f(t) = \sum_{n \in \mathbb{Z}} a_n \mathrm{e}^{i\lambda_n t}, \quad (a_n) \in \ell^2, \lambda_n \in \Lambda, 
$$
where the frequency set $\Lambda = \{\lambda_n\}_{n \in \mathcal{I}}$ differs from the classical $\{2\pi n/T\}_{n \in \mathbb{Z}}$. When the frequency set is non-harmonic, as is the case with $\{\omega_{m,n}\}$, we cannot immediately apply Fourier inversion. But under certain conditions, we can get a result similar to Parseval's identity: Ingham's theorem (cf. \cite{observation}) states that if 
${ \inf_{m \neq n} |\lambda_n - \lambda_m| = \gamma > 0 }$ and $ T > 2 \pi /\gamma$, then there are positive constants $c_-,c_+$ such that
$$c_-\sum_{n} |a_n|^2 \leq \int_0^T \big|\sum_n a_n e^{i\lambda_n t} \big|^2 \mathrm{d}t \leq c_+ \sum_{n} |a_n|^2.$$
For fixed $m$ (or $n$), the set $\{\omega_{m,n}\}_{n}$ satisfies the requirements of Ingham's theorem, and we use this is \Cref{sec:stability-finite}. However, it is easy to check that $|\omega_{m,0} - \omega_{m,1}| \to 0$ as $m \to \infty.$, i.e., if both indices are allowed to vary, the gap condition does not hold. To circumvent this, we will use a weaker version of Ingham's theorem from \cite{loreti1997partial}, that was used to prove observability of the wave equation in \cite{mehrenberger2009ingham,komornik2014cross}. 

\begin{theorem}
(Weakened Ingham's Theorem((Proposition 3.1 in \cite{komornik2014cross})
Let $\{\lambda_k\}_{k\in \mathbb{Z}}$ be a sequence of real numbers and $N$ a positive integer, and assume there exists some $\gamma > 0$ such that $$|\lambda_k - \lambda_l| \geq |k -l|\gamma \quad \text{if} \quad \max(|k|,|l|) \geq N.$$ 
Then for any $(a_n) \in \ell^2$ and $T > 0$,  
\begin{equation}
    \int_0^T \big|\sum_{k\in \mathbb{Z}}a_k e^{i\lambda_k t} \big|^2 \mathrm{d}t \geq \frac{2T}{\pi}\left(\sum_{|k| \geq N}|a_k|^2 - \left(\frac{2\pi}{T\gamma}\right)^2\sum_{k \in \mathbb{Z}}|a_k|^2 \right). 
    \label{weakham}
\end{equation}
\label{weakingham}
\end{theorem}
We now adapt the observability proof from \cite{mehrenberger2009ingham,komornik2014cross} to water waves. To show that \cref{weakingham} is satisfied for $f_m(t)$
in \eqref{Mexp}, we note that with   $$ \lambda_k^m = \begin{cases} \text{sgn}(k)\omega_{m,|k|-1} ,\quad |k| > 0,  \\ 0, \quad k = 0, \end{cases} \quad \text{and}\quad  a_k^m = \begin{cases}
    c_{m,k-1},  \quad k > 0 \\
    \bar{c}_{m,|k|-1}, \quad k < 0, \\
    0, \quad k = 0, 
\end{cases}  
\quad \text{for } k \in \mathbb{Z}, $$
we have 
$$ \sum_{k \in \mathbb{Z}} a_k^m e^{i\lambda_k^m t} = f_m(t). $$ Hence, if the gap condition in \cref{weakingham} holds for the frequency set $ \{\lambda_{k}^m\}_{k\in \mathbb{Z} }$, then \eqref{weakham} holds for $f_m(t)$.

\begin{lemma}
For any $m,n \in \mathbb{N}_0$, let  $$\{\lambda_k^m = \text{sgn}(k)\omega_{m,|k|-1}\}_{k \in \mathbb{Z}} \quad  \text{and} \quad  \{\tilde{\lambda}_k^n = \text{sgn}(k)\omega_{|k|-1,n}\}_{k \in \mathbb{Z}} .$$ 
It then holds for  $k,l \in \mathbb{Z} $  that 
\begin{align*}
    |\lambda_{k}^m - \lambda_{l}^m| &\geq |k -l| \gamma, \quad \text{for } \max(|k|,|l|) \geq m, \\
    |\tilde{\lambda}_{k}^n - \tilde{\lambda}_{l}^n| &\geq |k -l|\gamma, \quad \text{for } \max(|k|,|l|) \geq n, 
\end{align*}
with  
\begin{equation*}
    \gamma = \frac{\pi W}{2L\sqrt{L^{2} + W^{2}} }\sqrt{\frac{Hg}{2\sqrt{1+\frac{H^2g}{3S}}} - \frac{3S}{2H}}. 
    \label{gapconst}
\end{equation*} 
%\frac{M}{W}\left(\sqrt{1 + (W/LM)^2} - 1\right)\ &\frac{\pi}{L}\sqrt{\frac{Hg}{2\sqrt{1+\frac{H^2g}{3S}}} - \frac{3S}{2H}}, \quad M = 0,

\label{gap2}
\end{lemma}
\begin{proof} 
If $\text{sgn}(k) = \text{sgn}(l)$, then $|\lambda^m_k - \lambda^m_l| = |\omega_{m,|k|-1} - \omega_{m,|l|-1}| $. On the other hand, for $k,l > 0$, 
$$ \lambda_k^m  - \lambda_{-l}^m = \omega_{m,k-1} + \omega_{m,l-1} \geq \omega_{m,k-1} - \omega_{m,l-1}, $$
and so if for any $p,q \geq 0$ it holds that 
\begin{equation*}
    |\omega_{m,p} - \omega_{m,q}| \geq |p - q| \gamma \quad \text{and} \quad \omega_{m,p} \geq p\gamma \quad \text{for } \max(p,q) \geq m, 
\end{equation*}
then the gap condition holds for $\{\lambda_k\}_{k\in \mathbb{Z}\setminus \{0\}}$. \\

We now consider the function $G(t) = \sqrt{(gt +S t^3)\tanh(tH)} $ for $t > 0$. We have 
\begin{equation*}
    G'(t) = \frac{(g+3S t^2)\tanh(tH) + H(gt +S t^3)\text{sech}(tH)^2}{2\sqrt{(gt +S t^3)\tanh(tH)}}. 
\end{equation*}
Using elementary inequalities together with the lower bound $  \tanh(t) > t/\sqrt{1 +t^2} , t >0, $ (see Lemma 4 in \cite{bagul2021tight}), we find that
$$ G'(t) >\sqrt{\frac{H(g+3St^2)}{4(1+Ht)}} \geq \sqrt{\frac{Hg}{2\sqrt{1+\frac{H^2g}{3S}}} - \frac{3S}{2H}}  = \tilde{\gamma} $$
Fixing $m$ and assuming $p \geq m$ and $ p > q \geq  0$, we have that 
$$ k_{m,p} \leq \pi p\sqrt{L^{-2} + W^{-2}} < \pi(p + q)\sqrt{L^{-2} + W^{-2}}  \quad \text{and} \quad   k_{m,p}^2 - k_{m,q}^2 = \frac{\pi^2}{W^2}( p^2 - q^2).$$
Hence 
$$ |k_{m,p} - k_{m,q}| = \frac{|k_{m,p}^2 - k_{m,q}^2|}{|k_{m,p} + k_{m,q}|} > \frac{\frac{\pi^2}{W^2}|p - q||p+q|}{2\pi|p + q|\sqrt{L^{-2} + W^{-2}} } \geq \frac{\pi L}{2W\sqrt{L^{2} + W^{2}} }|p - q|. $$
We obtain
\begin{equation} 
\begin{split}
    |\omega_{m,p} - \omega_{m,q}| &= |G(k_{m,p}) - G(k_{m,q})| \\
    &= \bigg|\int_{k_{m,p}}^{k_{m,q}} G'(t) \mathrm{d}t\bigg| \geq \frac{\pi L}{2W\sqrt{L^{2} + W^{2}} }|p - q| \tilde{\gamma}.
\end{split}
\label{omega-dist}
\end{equation}
We also have that $k_{m,p} \geq \frac{\pi}{W}p$ and so $\omega_{m,p} \geq \frac{\pi}{W}p \tilde{\gamma}$. Arguing the same way, we find the estimate 
$$|\omega_{p,n} - \omega_{q,n}| \geq \frac{\pi W}{2L\sqrt{L^{2} + W^{2}} }|p - q| \tilde{\gamma}, $$
and since we have assumed $L\geq W$, the result follows. 
\end{proof}

We now prove \cref{stability}.

\begin{proof}
On the boundary $\Gamma_M$, we have 
\begin{align*}
    \mathcal{M}(X,t) |_{x_2 = 0} &= \sum_{m \geq 0}  \frac{C_{m,|n|}}{\sqrt{WL}} \cos(\pi m x_1/L) \sum_{n\in\mathbb{Z}}  a_k^n e^{i\lambda_n^m t}, \\
    \mathcal{M}(X,t)|_{x_1 = 0} &= \sum_{n \geq 0}  \frac{C_{|m|,n}}{\sqrt{WL}} \cos(\pi n x_2/L) \sum_{m \in \mathbb{Z}}b_m^n e^{i\tilde{\lambda}_{m}^n t},
\end{align*}
where $$b_m^n = \begin{cases}
    c_{m-1,n},  \quad m > 0 \\
    \bar{c}_{|m|-1,n}, \quad k m 0, \\
    0, \quad m = 0, 
\end{cases}  
\quad \text{for } k \in \mathbb{Z}. $$ We write
$$ I_1 = \int_0^T \int_0^L |\mathcal{M}(X,t)|^2 \mathrm{d}x_1 \mathrm{d}t \quad \text{and} \quad  I_2  = \int_0^T \int_0^W |\mathcal{M}(X,t)|^2 \mathrm{d}x_2 \mathrm{d}t. $$
By orthogonality,\cref{weakingham}, and \cref{gap2}, we have   
\begin{align*} 
I_1 = \frac{2}{W}\sum_{m \geq 0} \int_0^T \bigg|\sum_{n\in\mathbb{Z}}   a_n^m e^{i\lambda_n^m t} \bigg|^2 \mathrm{d}t 
\geq \frac{4T}{W\pi}\sum_{m \geq 0}  \left( \sum_{\substack{n\in\mathbb{Z} \\ |n| \geq m}} |a_n^m|^2 - \left(\frac{2 \pi}{T\gamma}\right)^2 \sum_{\substack{n\in\mathbb{Z} }} |a_n^m|^2\right)
\end{align*}
and
\begin{align*} 
I_2 = \frac{2}{L}\sum_{n \geq 0} \int_0^T \bigg|\sum_{m\in\mathbb{Z}}   b_m^n e^{i\tilde{\lambda}_m^n t} \bigg|^2 \mathrm{d}t 
\geq \frac{4T}{L\pi}\sum_{n \geq 0}  \left( \sum_{\substack{m \in\mathbb{Z} \\ |m| \geq n }} |b_m^n|^2 - \left(\frac{2 \pi}{T\gamma}\right)^2 \sum_{m\in\mathbb{Z}} |b_m^n|^2\right).
\end{align*}
Consequently,  
\begin{align*}
    I_1/L + I_2/W &\geq \frac{4T}{WL\pi}\left( \sum_{m \geq 0} \sum_{\substack{n\in\mathbb{Z} \\ |n| \geq m}} |a_n^m|^2 +\sum_{n \geq 0} \sum_{\substack{m\in\mathbb{Z} \\ |m| \geq n}} |b_m^n|^2   - \left(\frac{2 \pi}{T\gamma}\right)^2\sum_{k \geq 0} \sum_{\substack{l\in\mathbb{Z} }} |a_l^k|^2 + |b_l^k|^2 \right).
\end{align*}
In terms of $c_{m,n}$, we have that 
\begin{align*}
    \sum_{m \geq 0} \sum_{\substack{n\in\mathbb{Z} \\ |n| \geq m}} |a_n^m|^2 +\sum_{n \geq 0} \sum_{\substack{m\in\mathbb{Z} \\ |m| \geq n}} |b_m^n|^2& = 2 \sum_{m \geq 0} \sum_{\substack{n\in\mathbb{N} \\ n \geq m}} |c_{m,n-1}|^2 +2\sum_{n \geq 0} \sum_{\substack{m\in\mathbb{N} \\ m \geq n}} |c_{m-1,n}|^2 \\ &\geq 2\sum_{m \geq 0 } \sum_{n \geq 0} |c_{m,n}|^2,  
\end{align*}
and 
\begin{align*}
    \sum_{k \geq 0} \sum_{\substack{l\in\mathbb{Z} }} |a_l^k|^2 + |b_l^k|^2 = 2\sum_{m \geq 0 } \sum_{n \geq 0} |c_{m,n}|^2.
\end{align*}
As a result,
\begin{align*}
    I_1/L + I_2/W  & \geq  \frac{8T}{WL\pi}\left( \left(1  - \left(\frac{2 \pi}{T\gamma}\right)^2\right)\sum_{m \geq 0} \sum_{n \geq 0} |c_{m,n}|^2 \right) \\ 
    &= \frac{8T}{WL\pi}\left( \left(1  - \left(\frac{2 \pi}{T\gamma}\right)^2\right)\sum_{m \geq 0} \sum_{ n \geq 0} q_{m,n}^2 + \left(\frac{\omega_{m,n} }{g + Sk_{m,n}^2} p_{m,n}\right)^2 \right) 
\end{align*}
Since $\frac{\omega_{m,n}}{g + Sk_{m,n}^2} \sim (m^2+n^2)^{-1/4}$, there is a constant $C_s$ such that $$\sum_{m \geq 0} \sum_{ n \geq 0} q_{m,n}^2 + \left(\frac{\omega_{m,n} }{g + Sk_{m,n}^2} p_{m,n}\right)^2  \geq C_s\left(\|\eta_0\|_{L^2}^2 + \|\varphi_0\|_{\dot{H}^{-1/2}}^2\right).$$
With $T$ as in theorem we get $1  - \left(\frac{2 \pi}{T\gamma}\right)^2 = 3/4$. Then, with $C = \frac{2}{L\pi}C_s $, we have 
\begin{equation*}
    \int_0^T \int_{\Gamma_M} |\mathcal{M}(X,t)|^2 \mathrm{d}X\mathrm{d}t = I_1 + I_2 \geq TC \left(\|\eta_0\|_{L^2}^2 + \|\varphi_0\|_{\dot{H}^{-1/2}}^2\right).
\end{equation*}
\end{proof}
\begin{proof}
(Proof of  \cref{increased stability}) For $k = 1,2,...$ we have that $ \omega_{m,n}^{2k} \sim (m^2+n^2)^{3k/2}$. 
Assuming that the initial data is sufficiently smooth, we differentiate $\mathcal{M}$ with respect to time $k$ times and carry out the argument above to find   
\begin{align*}
    I_1/L + I_2/W  = \frac{8T}{WL\pi}\left( \left(1  - \left(\frac{2 \pi}{T\gamma}\right)^2\right)\sum_{m \geq 0} \sum_{ n \geq 0} q_{m,n}^2\omega_{m,n}^{2k} + \left(\frac{\omega_{m,n}^{1 + k} }{g + Sk_{m,n}^2} p_{m,n}\right)^2 \right). 
\end{align*}
The conclusion now follows since there is a constant $C_k$ such that 
$$\sum_{m \geq 0} \sum_{ n \geq 0} q_{m,n}^2\omega_{m,n}^{2k} + \left(\frac{\omega_{m,n}^{1+k} }{g + Sk_{m,n}^2} p_{m,n}\right)^2  \geq C_k\left(\|\eta_0\|_{\dot{H}^{\frac{3}{2}k}}^2 + \|\varphi_0\|_{\dot{H}^{\frac{3}{2}k-1/2}}^2\right).$$

\end{proof}

\subsection{Proof of \cref{stability-finite}}
The idea of this proof is to construct a well-conditioned mapping from the discrete measurement to the Fourier coefficients of the initial disturbance. It is achieved by showing that a finite-dimensional version of Ingham's theorem holds under the right conditions.  
\begin{proof}
Set $\mathcal{A} = [a_n^m]$ and $\mathcal{C} = [C_{m,|n|}a_n^m]$ the matrix of the Fourier coefficients of $P_\beta\eta_0$, where $C_{m,|n|}$ are the normalisation constants for $\psi_{m,n}$, and $\|P_\beta \eta_0 \|_{L^2} = \|\mathcal{A}\|_F$.  
We have that 
$$   P_\beta\eta(X,t)|_{\Gamma_M} = \sum_{ 0 \leq m \leq \beta}  \cos(\pi m x_1/L) \sum_{0 < |n|\leq \beta+1}C_{m,|n|}   a_n^m e^{i\lambda_n^m t}, $$
and for a fixed time $t_j$, we get 
$$\mathcal{M}_\beta(k,j) = \sum_{ 0 \leq m \leq \beta}  \cos(\pi m (k+1/2)/N_x) \sum_{0 <|n|\leq \beta +1 } C_{m,|n|}  a_n^m e^{i\lambda_n^m t}, \quad 0 \leq k \leq N_x-1.$$
Writing $$u_j = \begin{bmatrix} \sum_{0< |n|\leq \beta+1}  a_n^0 e^{i\lambda_n^m t_j} \\ \sum_{0<|n|\leq \beta+1}  a_n^1 e^{i\lambda_n^1 t_j}
        \\ \vdots \\ \sum_{0<|n|\leq \beta +1}  a_n^\beta e^{i\lambda_n^\beta t_j},
        \end{bmatrix},$$
we see that the measurement at time $t_j$ is given by\footnote{We use the notation $V(:,k)$ and $V(n,:)$ to denote the $k'$th column and $n'$th row of the matrix $V$, respectively.} $\mathcal{M}_\beta(:,j)=  D_3 u_j$, where 
$D_3$ is the discrete cosine transform matrix of the third type (cf. \cite{strang1999discrete}), i.e., 
\begin{equation*}
    D_3 = [ \cos(\pi(k+1/2)m/N_x) ] \in \mathbb{R}^{N_x\times N_x}, \quad 0 \leq k \leq N_x-1, \quad  0\leq m \leq N_x-1.
\end{equation*}
The matrix $D_3$ is orthogonal and can be made unitary by scaling the first row by $1/\sqrt{N}$ and the others by $\sqrt{2/N}$. Writing $U = [u_0 \enspace  u_1 \enspace \hdots \enspace  u_{N_t} ]$, we have 
$$ \mathcal{M}_\beta = D_3 U \implies U = D_3^{-1} \mathcal{M}_\beta.$$
The point of this is that the row $U(m,:)$ contains the sampled values of $\sum_{0<|n|\leq \beta +1}  a_n^m e^{i\lambda_n^m t}$. 
Recalling that $\lambda_n^m = \text{sgn}(n)\omega_{m,|n|-1}$, we define the non-harmonic Vandermonde matrix $V_m$ as follows\footnote{The indexing is awkward, since $\omega_{0,0} = 0$ is included. But since $p_{0,0} =(\eta_0,\varphi_{0,0})_{L^2} = 0$, we can just remove the columns with $\omega_{0,0}$-terms from $V_0$. We will assume that and not mention it in what follows.}: let $v_n^m =e^{i\omega_{m,n}\Delta t}$ and
\begin{equation}
    V_m(j,:) = 
    \begin{bmatrix}
        (\overline{v_\beta^m})^j \quad   \cdots \quad  (\overline{v_0^m})^j \quad (v_0^m)^j  \quad \cdots \quad (v_\beta^m)^j
    \end{bmatrix},
     \quad \text{for } 0 \leq j \leq N_t. 
    \label{vandermonde}
\end{equation}
In terms of $V_m$, we can write
\begin{equation*}
    U(m,:)^\top = V_m \mathcal{C}(m,:)^\top. 
\end{equation*}
To show that $V_m$ is well-conditioned with the choice of $T$ and $N_t$, we rely on a type of discrete Ingham's inequality (cf. \cite{aubel2019vandermonde} or Theorem 10.23 \cite{plonka2018numerical}).  
We set $\{\xi_n\}_{n \leq N_\beta } \subset [-1/2,1/2)$, and let
\begin{equation}
    \tilde{\delta} = \min_{k\neq l} \left(\min_{n \in \mathbb{Z}} | \xi_k - \xi_l -n|\right).
    \label{wrap around}
\end{equation}
Let $v_n = \exp(i2\pi\xi_n)$, and let $N_t > \max(N_\beta ,1/\tilde{\delta})$. Then, the Vandermonde matrix
$$V = [v_n^j], \quad 0 \leq j \leq N_t,  \quad 0 \leq n \leq N_\beta,$$
satisfies 
$$ (N_t - 1/\tilde{\delta}) \|x\|_2^2\leq \|Vx\|_2^2 \leq (N_t + 1/\tilde{\delta})\|x\|_2^2, \quad \text{for all } x \in \mathbb{C}^{N_\beta}.$$
We now show that the under the assumptions on $N_t$ and $T$ in  \Cref{stability-finite}, the matrix $V_m$ defined in \eqref{vandermonde} 
satisfies the conditions of the discrete Ingham's inequality. For fixed $m \geq 0$, we have $$\{\xi_n\} = \{-\omega_{m,n}\Delta t/2\pi\}_{n \leq \beta} \cup \{\omega_{m,n}\Delta t2/\pi\}_{n \leq \beta}.  $$ 
The so-called ``wrap-around" metric in \eqref{wrap around} ensures that the matrix entries $e^{i2\pi \xi_n}$ are separated. For example, for $\xi_0 = 0, \xi_1 = 0.75, \xi_2 = 1 $, we have $\min_{n \in \mathbb{Z}} | \xi_0 - \xi_2 -n| = 0$ and ${\min_{n \in \mathbb{Z}} | \xi_0 - \xi_1 -n| = 0.25}$. We choose $\Delta t = \frac{4\pi}{5 \omega_{\beta,\beta}}$, as this ensures that 
$\{\omega_{m,n}\Delta t/2\pi\}_{|n|\leq \beta} \subset [-2/5,2/5]$ for all $m \leq \beta$. From  \cref{gap2} we find a bound on the distance between $\omega_{m,n}$ for varying $n$ and fixed $m$: 
Assume $ n,p \geq 0$. From  equation \eqref{omega-dist} in  \cref{gap2}, we have that $|\omega_{m,n} - \omega_{m,p}| \geq |k_{m,n} - k_{m,p}|\tilde{\gamma}$. For $m > 0$, 
\begin{equation*}
\begin{split}
    \min_{n \neq p} |k_{m,n} - k_{m,p}| &= |k_{m,1} - k_{m,0}| 
    = \pi\frac{m}{L}\left(\sqrt{1 + (L/Wm)^2)} -1\right)  \\
    &\geq \pi\frac{\beta}{L}\left(\sqrt{1 + (L/W\beta)^2)} -1\right)  \geq  \frac{\pi}{L}\left(\sqrt{\beta^2 + 1} -\beta\right).
\end{split}
\label{dist_est}
\end{equation*} 
Since $\pi/W \geq \frac{\pi}{L}\left(\sqrt{\beta^2 + 1} -\beta\right)$, this estimate also holds for $m = 0$. Finally, the smallest distance between two frequencies with opposite signs satisfies  $2\omega_{1,0} \geq k_{1,0}\tilde{\gamma}$, and we get a lower bound on the minimal wrap-around distance:
$$ \delta = \min_{k\neq l} \left(\min_{n \in \mathbb{Z}} |  \omega_{m,k}\frac{\Delta t}{2\pi} -  \omega_{m,l}\frac{\Delta t}{2\pi}-n|\right) \geq \frac{\Delta t}{2\pi} \min\left(1/5,\frac{\pi}{L}\left(\sqrt{\beta^2 + 1} -\beta\right) \tilde{\gamma}\right). $$

Recalling the choice of $ N_t = \bigg\lceil 1 + \frac{1}{\delta} \bigg\rceil$ and $T = \frac{4\pi N_t}{5\omega_{\beta,\beta}}$, we have $\Delta t = \frac{T}{N_t} = \frac{4\pi}{5 \omega_{\beta,\beta}}$. Hence, with $N_\beta = 2\beta +2$, the conditions of discrete Ingham's inequality hold. For each $m$, we therefore have 
\begin{align*}
    \left(N_t-\frac{1}{\delta}\right)\|\mathcal{C}(m,:)\|_2^2 &\leq \|V_m \mathcal{C}(m,:)^\top\|_2^2  = \|U(m,:)\|_2^2 \\
     & =  \|\left(D_3^{-1}\mathcal{M}\right)^\top(m,:)\|_2^2 \leq \|D_3^{-1}\|_2^2\|\mathcal{M}(:,m)\|_2^2
\end{align*}
Taking $S = \text{diag}(1/\sqrt{N_x},\sqrt{2/N_x},...,\sqrt{2/N_x})$, we have that $SD_3$ is unitary, and so ${D_3^{-1} = D_3^\top S^2}$. Consequently, we find that  $$\|D_3^{-1}\|_2 = \sqrt{\lambda_{\text{max}}(D_3^{-1}(D_3^{-1})^\top)} = \sqrt{\lambda_{\text{max}}(D_3^\top S(S^2) SD_3)}= \sqrt{\frac{2}{N_x}}.$$ Hence 
\begin{equation*}
    \|\mathcal{C}\|_F^2 = \sum_m \|\mathcal{C}(m,:)\|_2^2 \leq \sum_m \frac{2/N_x}{\left(N_t-\frac{1}{\delta}\right)}\|\mathcal{M}_\beta(:,m)\|_2^2 \leq \frac{2}{N_x\left(N_t-\frac{1}{\delta}\right)} \|\mathcal{M}_\beta\|_F^2, 
\end{equation*}
and since $\|\mathcal{A}\|_F^2 \leq \frac{4}{LW}\|\mathcal{C}\|_F^2$, we get 
$$ \|\mathcal{A}\|_F^2 \leq  \frac{LW}{2N_x\left(N_t-\frac{1}{\delta}\right)} \|\mathcal{M}_\beta\|_F^2. $$
Denote by $\mathcal{A}_R$ the matrix of Fourier coefficients of the reconstruction $R_\beta \eta_0$ obtained by applying procedure outlined above\footnote{See \Cref{section: experiments}.} to $\mathcal{M}$ . It now follows by linearity that 
 $$\|P_\beta \eta_0 - R_\beta \eta_0 \|_{L^2}^2 = \|\mathcal{A} - \mathcal{A}_R\|_F^2 \leq \frac{LW}{2N_x\left(N_t-\frac{1}{\delta}\right)} \|\mathcal{M}_\beta - \mathcal{M}\|_F^2.$$ 
Last, for $f \in \dot{H}^s, s > 0$, we have that 
\begin{equation*}
    \|f -P_\beta f\|_{L^2}^2 = \sum_{m,n > \beta} \hat{f}_{m,n}^2 \leq 2\beta^{-2s}\sum_{m,n > \beta}(m^2 + n^2)^s\hat{f}_{m,n}^2, 
\end{equation*}
and hence 
\begin{equation*}
    \|f -P_\beta f\|_{L^2} \leq \frac{\sqrt{2}}{\beta^s} \|f\|_{\dot{H}^s}. 
\end{equation*}
Consequently, we get 
\begin{align*}
    \|\eta_0 - R_\beta \eta_0 \|_{L^2} &= \|\eta_0 - P_\beta \eta_0 + P_\beta \eta_ 0 - R_\beta \eta_0 \|_{L^2} \leq \|\eta_0 - P_\beta \eta_0 \|_{L^2} + \|P_\beta \eta_ 0 - R_\beta \eta_0 \|_{L^2} \\
        &\leq  \frac{\sqrt{2}}{\beta^s} \|f\|_{\dot{H}^s} + \sqrt{\frac{LW}{2N_x\left(N_t-\frac{1}{\delta}\right)}} \|\mathcal{M}_\beta - \mathcal{M}\|_F.
\end{align*}

\end{proof}

\bibliographystyle{siamplain}
\bibliography{references}
\end{document}